\documentclass[a4paper,italian,english,12pt]{amsart}
\usepackage{babel}

\usepackage{graphicx,epsfig,psfrag}

\usepackage{amsmath}
\usepackage{amsthm}
\usepackage{amssymb}

\usepackage{bbold}
\usepackage{tikz-cd}

\usepackage[colorlinks,citecolor=blue,urlcolor=blue,bookmarks=true]{hyperref}
\hypersetup{
pdfpagemode=UseNone,
pdfstartview=FitH,
pdfdisplaydoctitle=true,
pdfborder={0 0 0}, 
pdftitle={Free group representations: duplicity on the boundary},
pdfauthor={Waldemar Hebisch, Gabriella Kuhn, and Tim Steger}
pdflang=en-US
}

\newcommand{\ID}{\operatorname{Id}}
\newcommand{\tr}{\operatorname{tr}}
\newcommand{\CC}{\mathbf{C}}
\newcommand{\RR}{\mathbf{R}}

\newcommand{\mcH}{\mathcal{H}}
\newcommand{\mcT}{\mathcal{T}}

\newcommand{\mcU}{\mathcal{U}}

\newcommand{\mcB}{\mathcal{B}}

\newcommand{\one}{\mathbf{1}}

\newcommand{\mcA}{\mathcal{A}}
\newcommand{\GG}{\Gamma}
\newcommand{\Gt}{\tilde{G}}
\newcommand{\bGG}{\partial\Gamma}
\newcommand{\cp}{\Gamma\ltimes C(\bGG)}
\newcommand{\cpK}{\Gamma\ltimes C(K)}

\newcommand{\Aut}{\operatorname{Aut}}
\DeclareMathOperator{\TR}{TR}
\DeclareMathOperator*{\wklim}{\textup{wk-lim}}
\newcommand{\norm}[1]{\|{#1}\|}

\newcommand{\st}{\,;\;}
\newcommand{\period}{\;.}
\newcommand{\comma}{\;,}
\newcommand{\defn}[1]{\emph{#1}}
\newcommand{\inv}[1]{{#1}^{-1}}

\newcommand{\pir}{\pi_{\textrm{reg}}}
\newcommand{\positivi}{\mcB^+(\mcH)}
\newtheorem{theorem}{Theorem}[section]
\newtheorem{proposition}[theorem]{Proposition}
\newtheorem{lemma}[theorem]{Lemma}
\newtheorem{corollary}[theorem]{Corollary}
\newtheorem*{mthm}{Oddity Theorem Using~(FTC)}
\newtheorem*{dthm}{Duplicity Theorem Using~(FTC)}

\theoremstyle{definition}
\newtheorem{definition}[theorem]{Definition}
\newtheorem{remark}[theorem]{Remark}

\newenvironment{pmat}{\begin{pmatrix}}{\end{pmatrix}}

\begin{document}
\selectlanguage{english}

\title[Duplicity on the boundary]
	{Free group representations: duplicity on the boundary}

\author{Waldemar Hebisch}
\address{Waldemar Hebisch \\Mathematical Institute \\
University of Wroc{\l}aw \\
50-384 Wroc{\l}aw \\
pl. Grunwaldzki 2/4 \\
POLAND}
\email{hebisch@math.uni.wroc.pl}
\author{Gabriella Kuhn}
\address{Gabriella Kuhn \\
Dipartimento di Matematica e Applicazioni \\
Universit\`a di Milano Bicocca \\
via Cozzi 53 \\
20125 Milano \\
ITALY}
\email{mariagabriella.kuhn@unimib.it}

\author{Tim Steger}
\address{Tim Steger\\
		 Matematica\\
		 Universit\`a degli Studi di Sassari\\
		 Via Piandanna 4\\
		 07100 Sassari\\
		 Italy}
\email{steger@uniss.it}

\pagestyle{myheadings}

\begin{abstract}
We present a powerful theorem for proving the irreducibility of
tempered unitary representations of the free group.
\end{abstract}

\keywords{free group, crossed-product $C^*$-algebra, irreducible,
  tempered, unitary representation}

\subjclass{Primary; 22D10, 43A65. Secondary: 15A48, 22E45, 22E40}

\maketitle

\section{Introduction}\label{intro0}
Let $\Gamma$~be a finitely generated non-abelian free group and let
$\pi:\Gamma\to\mcU(\mcH)$ be a unitary representation of $\Gamma$.
Let~$\bGG$ be the usual boundary of~$\Gamma$. When~$\pi$ is
\defn{tempered}, that is when it is weakly contained in the regular
representation, one can sometimes view~$\mcH$ as an $L^2$-space
on~$\bGG$, where the action of~$x\in\Gamma$ on an $L^2$-function is
given by a two-stage operation: first, translate the function via~$x$;
second, apply some pointwise linear operation. Indeed, one can always
view~$\mcH$ as a subspace of an $L^2$-space with such an action. See
Proposition~\ref{tempered}.

Now suppose that $\pi$~is irreducible. There are lots of examples
where there are precisely two (essentially different) ways in which
$\mcH$~can be identified with an $L^2$-space on~$\bGG$. Roughly
speaking, this is the phenomenon which we call~$\emph{duplicity}$. The
single most important result here, the Duplicity Theorem, starts with
the hypothesis that we have two such identifications. There is a
further technical hypothesis, a Finite Trace Condition, which holds
for many interesting examples and fails for many others. Our first
main conclusion is that there are no identifications beyond the two we
started with.

It is not necessary to suppose that $\pi$~is irreducible.  Alternative
hypotheses, much easier to prove, give irreducibility as a second
main conclusion. See~\cite{Pe-S} for an application of this
technique. The biggest known family of examples where the hypotheses
and conclusions of the Duplicity Theorem hold is a certain subfamily
of the representations described in~\cite{K-S3}. We know of no other
method to prove their irreducibility in a uniform manner.

A third conclusion, under the same hypotheses, is an analogue of Schur
orthogonality. In the formula below $A_\pi$ is a positive constant,
$|x|$~stands for the word-length of~$x\in\Gamma$, $v_1,v_2\in\mcH$,
while $v_3$ and $v_4$ must be chosen in a certain dense subspace
$\mcH^\infty\subset\mcH$.
\begin{equation}\label{schur0}
  \lim_{\epsilon\to 0+} \epsilon
    \sum_{x\in\Gamma} e^{-\epsilon|x|}
      \langle v_1,\pi(x)v_3 \rangle
      \overline{\langle v_2,\pi(x)v_4 \rangle}
      =2A_\pi\langle v_1,v_2\rangle
        \overline{\langle v_3,v_4\rangle}
      \period
\end{equation}
Theorem~\ref{Schur} gives a more elaborate version of this identity
which connects limits of this sort with the two identifications
of~$\mcH$ with $L^2$-spaces on $\bGG$.

Besides duplicity, there are examples of irreducible~$\pi$
illustrating two other phenomena: \defn{monotony}, where there is only
one identification between~$\mcH$ and an $L^2$-space on~$\bGG$ and
\defn{oddity}, where there is only one identification, but it has to
be with a proper subspace of the $L^2$-space. This paper has nothing to
say about monotony, but there is an Oddity Theorem, closely analogous
to the Duplicity Theorem.

\section{Definitions and statements of results}\label{intro}
So let $\Gamma$~be a non-abelian free group on a given finite set of
free generators. Let $A\subseteq\Gamma$ consist of those generators
and their inverses. Let also $\pi:\Gamma\to\mcU(\mcH)$ be a unitary
representation of $\Gamma$.

Without going into the matter in any detail, we warn the reader that
$\Gamma$~is \emph{not} a Type~I group (see \cite{Dix}). Among other
things, this means that a given unitary representation may be
decomposable as a direct integral of irreducibles in more than one
way. Also, the unitary dual of~$\Gamma$, the space of equivalence
classes of unitary irreducibles, cannot be parametrized by any
standard Borel space (see \cite{Hjorth} and \cite{Glimm}) which means
in practice that one cannot hope for a parametrization that one could
actually work with. Moreover, the usual machinery of character theory
is not applicable.

So what can one do? Many papers construct specific families of
 representations and prove them irreducible. See for
example \cite{Yos}, \cite{Py-Sw}, \cite{FT-P}, \cite{FT-S},
\cite{Pe-S}, \cite{K-S1}, \cite{P}, \cite{K-S3}, and \cite{BG}.
Some of these papers also prove inequivalence of representations,
either within or between families. The first objective of this paper
is to explain a powerful indirect method for proving irreducibility
and inequivalence.

All the representations which will be of interest here have
``realizations'' as $L^2$-spaces on the \emph{boundary}
of~$\Gamma$. Recall that the Cayley graph of~$\Gamma$ with respect
to~$A$ is a tree, and that this tree has a standard compactification
which is obtained by adjoining a boundary, which we
denote~$\bGG$. This boundary can be described as the space of ends of
the tree; it also coincides with the boundary of~$\Gamma$ considered
as a Gromov hyperbolic group.  Concretely, if we identify~$\Gamma$ with
the set of finite reduced words:
\begin{equation*}
  \{ a_1a_2\dots a_n \st a_j\in A, a_ja_{j+1}\neq 1\}
\end{equation*}
then we can identify~$\bGG$ with the set of infinite reduced words
\begin{equation*}
  \{ a_1a_2a_3\dots  \st a_j\in A, a_ja_{j+1}\neq 1\}
    \period
\end{equation*}
For $x\in\Gamma$ let~$\Gamma(x)$ be the set of finite reduced words
which start with the reduced word for~$x$; let~$\bGG(x)$ be the set of
infinite reduced words which start with the reduced word for~$x$. A
basis for the topology on the compactification $\Gamma\sqcup\bGG$ is
given by the singletons $\{x\}$ and the sets $\Gamma(x)\sqcup\bGG(x)$,
as $x$~varies through~$\Gamma$.  The left-action of~$\Gamma$ on
$\Gamma$ extends to a continuous action on the compactification.

Let~$C(\bGG)$ be the commutative $C^*$-algebra 
of continuous complex valued functions on $\bGG$.
Likewise for~$C(\Gamma\sqcup\bGG)$.  If
one wishes to identify an abstract Hilbert space~$\mcH'$ with an
$L^2$-space on $\bGG$, the essence of the identication is given by the
action of~$C(\bGG)$ on~$\mcH'$ corresponding to pointwise
multiplication. This action exists no matter what measure on~$\bGG$ is
used to construct the $L^2$-space; also the $L^2$-space might be
vector-valued rather than scalar-valued; indeed the dimension of the
vectors might vary in some measurable way from point to point
of~$\bGG$. It is the \emph{spectral theorem} for~$C(\bGG)$ (see
\cite{Rudin}) which tells us that any $C(\bGG)$-action on~$\mcH'$ does
indeed correspond to an identification of~$\mcH'$ with an $L^2$-space
on~$\bGG$.

Obviously, a Hilbert space with no further structure can be identified
with an $L^2$-space on~$\bGG$ in many, many different ways. Now
suppose the Hilbert space carries a unitary representation,
$\pi':\Gamma\to\mcU(\mcH')$. We would like to identify~$\mcH'$ with an
$L^2$-space on~$\bGG$ in such a way that the $\Gamma$-action
on~$\mcH'$ matches up with the $\Gamma$-action
on~$\bGG$. Specifically, we would like the operator~$\pi'(x)$ to be a
two-stage operation as described in the introduction: first, translate
an $L^2$-function on~$\bGG$ via~$x$; second, apply some pointwise
linear operation.

While it is not hard to make this precise, it is more efficient to
express the concept in terms of the compatibility between the two
actions on~$\mcH'$: the action of~$\Gamma$ and the action
of~$C(\bGG)$. Denote both of these actions by~$\pi'$.  Let
$\lambda:\Gamma\to\Aut(C(\bGG))$ be given by:
\begin{equation*}
  (\lambda(x)G)(\omega)=G(x^{-1}\omega)
\end{equation*}
i.e. left-translation. The desired compatability is:
\begin{equation}\label{bdry-rpn}
  \pi'(x)\pi'(G)\pi'(x)^{-1} = \pi'(\lambda(x)G)
    \qquad\text{for $x\in\Gamma$ and $G\in C(\bGG)$.}
\end{equation}
In fact, a pair of actions which satisfy~\eqref{bdry-rpn} fit together
to give a representation of a certain
$C^*$-algebra, the \emph{crossed-product $C^*$-algebra},
denoted~$\cp$. Vice versa, any $\cp$-representation comes from a pair
of actions which fit together as per~\eqref{bdry-rpn}. The definition
of $\cp$ is standard, and can be found, for example,
in \cite{Davidson}, but there is also a short explanation in the
following section.

Given a unitary representation $\pi:\Gamma\to\mcU(\mcH)$, when is it
possible to identify~$\mcH$ with the representation space of a
$\cp$-representation? Quite often, as it happens, but to get a
clean answer we have to modify the question: when is it possible to
identify~$\mcH$ with a subspace of the representation space of a
$\cp$-representation?
\begin{definition}
  Let $\pi:\Gamma\to\mcU(\mcH)$ be a unitary representation
  of~$\Gamma$. A \defn{boundary realization} of~$\pi$ is an isometric
  $\Gamma$-inclusion $\iota$ of $\mcH$ into~$\mcH'$ where
  \begin{itemize}
  \item $\mcH'$ is the representation space of a
    $\cp$-representation~$\pi'$,
  \item and $\iota(\mcH)$ is cyclic for the action of~$C(\bGG)$
    on~$\mcH'$.
  \end{itemize}
\end{definition}

One thinks of the map~$\iota$ as an identification of~$\mcH$ with a
subspace of an $L^2$-space on~$\bGG$, where the $L^2$-space carries a
$\Gamma$-action compatible with the $\Gamma$-action on $\bGG$. If one
omits the second condition in the definition, and if one had a
boundary realization as above, then one could replace $\mcH'$ with
$\mcH'\oplus\mcH''$ and $\iota$ with $\iota\oplus 0$ for any second
$\cp$-representation space~$\mcH''$. It is convenient to exclude
this second, essentially irrelevant, summand.

\begin{proposition}\label{tempered}
  A unitary representation $\pi:\Gamma\to\mcU(\mcH)$ has a boundary
  realization if and only if $\pi$~is weakly contained in the regular
  representation of~$\Gamma$.
\end{proposition}

Suppose $\pi$~is \emph{irreducible} and weakly contained in the
regular representation. How many different boundary realizations does
it have?  To make this question precise, we need
\begin{definition}
  Let $\pi:\Gamma\to\mcU(\mcH)$ be a unitary representation
  of~$\Gamma$. Two boundary realizations $\iota_j:\mcH\to\mcH_j'$ are
  \defn{equivalent} if there exists a unitary map
  $J:\mcH_1'\to\mcH_2'$ between the two representation spaces which
  intertwines both the $\Gamma$-actions and the $C(\bGG)$-actions and
  such that $J\iota_1=\iota_2$.
\end{definition}

How many inequivalent boundary realizations does $\pi$~have? There are
many known examples where the answer is one; also many known examples
where the answer is not one. Call a boundary realization
$\iota:\mcH\to\mcH'$ a \defn{perfect boundary realization} if~$\iota$
is a unitary equivalence, i.e. a bijection and not just an injection.
In many of the known cases where $\pi$~has more than one boundary
realization, it has exactly two perfect boundary realizations and all
other boundary realizations are obtained as combinations of those two.
Indeed, given the known examples, one can conjecture that this is the
only possibility when $\pi$~has more than one realization.  See the
afterword to~\cite{K-S2} for a more detailed version of this
conjecture, and also for indications of how things stand for some
known families of representations.

The second main objective of this paper is to present a theorem which
(in many cases) allows one to prove, for a representation which has
two known perfect boundary realizations, that there are no others.

\begin{dthm}\label{DupThm} Let $\pi:\Gamma\to\mcU(\mcH)$ be a unitary
  representation of~$\Gamma$. Suppose
  \begin{itemize}
  \item $(\pi'_\pm,\mcH'_\pm)$ are two irreducible
    $\cp$-representations, inequivalent as $\cp$-representations.
  \item $\iota_\pm:\mcH\to\mcH'_\pm$ are two perfect boundary
    realizations of~$\pi$.
  \item The following Finite Trace Condition \eqref{FTC} holds
    \begin{equation*}
      \norm{(\iota_+^* \pi'(\one_{\bGG(a)})\iota_+)
      (\iota_-^* \pi'(\one_{\bGG(b)})\iota_-)}_{HS}<\infty
    \end{equation*}
     {for $a,b\in A$, $a\neq b$}.
  \end{itemize}
  Then
  \begin{itemize}
  \item $\pi$~is irreducible as a $\Gamma$-representation.
  \item Up to equivalence, $\iota_+$ and $\iota_-$ are the only perfect
    boundary realizations of~$\pi$.
  \item Any imperfect boundary realization of~$\pi$ is equivalent to
    the map
    $\sqrt{t_+}\,\iota_+\oplus\sqrt{t_-}\,\iota_-:\mcH\to\mcH'_+\oplus\mcH'_-$ 
    for constants $t_+,t_->0$ with $t_++t_-=1$.
  \end{itemize}
\end{dthm}

A representation~$\pi$ which satisfies the conclusions of this theorem
is said to satisfy \defn{duplicity}. Examples where both the
hypotheses and the conclusions are valid include the representations
of~\cite{Yos} and the non-endpoint representations
of~\cite{FT-P}. There are other examples of representations where the
\eqref{FTC} fails but which nonetheless satisfy duplicity.  There are further
examples without \eqref{FTC} where duplicity appears to hold, but for which
we have no proof.

The Duplicity Theorem is also a tool for proving the irreducibility
of~$\pi$.  To apply it, one must establish the irreducibility and
inequivalence of the two $\cp$-representations, $\pi'_\pm$, but
proving irreducibility for $\cp$-representations is far easier than
proving irreducibility for $\Gamma$-representations.

There are lots of examples of irreducible representations~$\pi$ which
have only one realization, that realization perfect:
\defn{monotony}. Proving this requires different methods than those
presented here. See \cite{K-S2}, \cite{KSS} and \cite{BG}.
On the other hand, representations which have a single
\emph{imperfect} realization can, if an appropriate Finite Trace
Condition holds, be attacked with the same methods as in the case of
duplicity.

\begin{mthm}\label{OddThm}
Let $\pi:\Gamma\to\mcU(\mcH)$ be a unitary
representation of~$\Gamma$. Suppose
  \begin{itemize}
  \item $(\pi',\mcH')$ is an irreducible $\cp$-representation.
  \item $\iota:\mcH\to\mcH'$ is an imperfect realization of~$\pi$.
  \item The following Finite Trace Condition \eqref{FTC} holds
    \begin{equation*}
      \norm{P_2\pi'(\one_{\bGG(a)})P_1}_{HS}<\infty
      \qquad\text{for each $a\in A$}
    \end{equation*}
     where $P_1:\mcH'\to\mcH'$ is the projection onto~$\iota(\mcH)$
     and $P_2=\ID-P_1$ is the projection onto the orthogonal
     complement of~$\iota(\mcH)$.
  \end{itemize}
  Then
  \begin{itemize}
  \item $\pi$~is irreducible as a $\Gamma$-representation.
  \item Up to equivalence, $\iota$ is the only boundary realization
    of~$\pi$.
  \end{itemize}
  Observe that the unitary $\Gamma$-action which $\pi'$ gives
  on~$\mcH'$ stabilizes $\mcH_1=\iota(\mcH)$, so it also stabilizes
  the orthogonal complement $\mcH_2=\mcH\ominus\mcH_1$. Let
  $\pi_2:\Gamma\to\mcU(\mcH_2)$ denote the $\Gamma$-action
  on~$\mcH_2$. One can also conclude:
  \begin{itemize}
  \item $\pi_2$ is irreducible.
  \item $\pi_2$ is inequivalent to $\pi$.
  \end{itemize}
\end{mthm}
A representation~$\pi$ which satisfies the conclusions of this theorem
is said to satisfy \defn{oddity}.  Among other examples, the
non-endpoint representations of~\cite{P} satisfy the hypotheses of the
theorem, and are examples of oddity. There are also known examples
where the \eqref{FTC} fails, but oddity holds nonetheless, and yet other known
examples where the \eqref{FTC} fails, and oddity appears to hold, but is not
proved.

The third main objective of this paper is to prove Schur orthogonality
relations, like~\eqref{schur0}.  For any $\cp$-representation~$\pi'$
and any function $G\in C(\Gamma\sqcup\bGG)$, let
$\pi'(G)=\pi'\bigl(G|_{\bGG}\bigr)$. Also, for $x\in\Gamma$, let
$G^*(x)=\bar G(x^{-1})$.
\begin{theorem}\label{Schur}
  Let $\pi$~be a representation of~$\Gamma$ on~$\mcH$ which satisfies
  the hypotheses of the Duplicity Theorem. There exists a constant
  $A_\pi>0$ and a dense subspace $\mcH^\infty\subset\mcH$ of
  \emph{good vectors} of~$\mcH$ so that for any $v_1,v_2\in\mcH$,
  $v_3,v_4\in\mcH^\infty$, and $G,\tilde G\in C(\Gamma\sqcup\bGG)$ the
  following holds
  \begin{multline}\label{schur1}
  \lim_{\epsilon\to 0+} \epsilon
    \sum_{x\in\Gamma} e^{-\epsilon|x|} G(x)\tilde G^*(x)
      \langle v_1,\pi(x)v_3 \rangle
      \overline{\langle v_2,\pi(x)v_4 \rangle} \\
      =A_\pi\left(
        \langle \pi'_+(G)\iota_+ v_1,\iota_+v_2\rangle
        \overline{\langle \pi'_-(\tilde G)
          \iota_-v_3,\iota_-v_4\rangle}\right. \\
      +\left.
        \langle \pi'_-(G)\iota_-v_1,\iota_-v_2\rangle
        \overline{\langle \pi'_+(\tilde G)\iota_+v_3,
           \iota_+v_4\rangle} \right)
      \period
  \end{multline}
\end{theorem}
On the right-hand side of~\eqref{schur1} one sees the two actions
of~$C(\bGG)$ on~$\mcH$, namely $\iota_\pm^*\pi'_\pm(\cdot)\iota_\pm$,
but the sum on the left-hand side is calculated using only the matrix
coefficients of the original~$\Gamma$-representation, $\pi$. If one
knew a priori what the space of good vectors was, \eqref{schur1} would
provide a canonical method for calculating the two boundary
realizations starting with~$\pi$. A slightly simpler formula holds for
representations satisfying the conditions of the Oddity Theorem.

Section~\ref{bdry-rlzn} discusses crossed-product algebras,
representations weakly contained in the regular representation, and
the proof of Proposition~\ref{tempered}. Section~\ref{imfS} introduces
some general machinery applicable to boundary
realizations. Section~\ref{inner-product} explains a certain inner
product which is used in the main proofs. Section~\ref{weak-limit}
discusses a certain limit closely related to the left-hand side
of~\eqref{schur1}. Section~\ref{good-vectors} discusses the subspace
$\mcH^\infty\subseteq\mcH$ of good vectors. Section~\ref{main-proofs}
has the proofs of the Duplicity Theorem and the Oddity
Theorem. Section~\ref{schur-sec} proves~\eqref{schur0}
and~\eqref{schur1}.

We follow the convention that a positive constant denoted by~$C$ may
change its exact value from one line to the next. In general, $\one_S$
is the characteristic function of the set~$S$. However, if
$x\in\Gamma$, we abbreviate and write $\one_x$ for the characteristic
function of~$\bGG(x)$. Simply~$\one$ usually stands for~$\one_{\bGG}$.
We assume that all our Hilbert spaces are separable.

\section{The crossed product $\cp$ and
  boundary realizations}\label{bdry-rlzn}

The reader who can do without the proof of Proposition~\ref{tempered}
can skip this section. Or one can read up through the definition of
crossed-product $C^*$-algebras, and skip the rest.

We said already that when a $\Gamma$-repre\-sentation and a
$C(\bGG)$-repre\-sentation act on the same Hilbert space and satisfy
the compatibility condition~\eqref{bdry-rpn}, then this pair of
representations can be thought of as a representation of a certain
crossed-product $C^*$-algebra, $\cp$. Here we define crossed-product
algebras and clarify the above assertion. After that we give some
further definitions. All of this is preparation for the proof of
Proposition~\ref{inclusion-in-boundary} which will give us half of
Proposition~\ref{tempered}.

Assume that $\GG$ acts on a $C^*$-algebra $\mcA$ by isometric
automorphisms 
$\lambda:\GG\to\operatorname{Aut}(\mcA)$.
\begin{definition}\label{covariant}
A \defn{covariant representation} of $(\GG,\mcA)$ on a Hilbert space
$\mcH$ is a triple $(\pi,\alpha,\mcH)$ where
\begin{itemize}
\item $\pi:\GG\to\mcU(\mcH)$ is a unitary representation of $\GG$,
\item $\alpha:\mcA\to \mcB(\mcH)$ is a $ ^*$-representation of $\mcA$,
\item $\pi(x)\alpha(G)\pi(x)^{-1}=\alpha(\lambda(x)G)$
    for every $G\in\mcA$ and $x\in \GG$.
\end{itemize}
\end{definition}
Let $\mcA[\Gamma]$ denote the space of finitely supported functions from
$\Gamma $ to $\mcA$:
\begin{equation*}
  \mcA[\GG]=
  \left\{\text{finite sums $\sum_i G_i\delta_{x_i}$}
    \st x_i\in\GG,\,G_i\in\mcA\right\}
\end{equation*}
where $\delta_x$ denotes the Kroneker function at $x\in\GG$.  We endow
$\mcA[\GG]$ with a $C^*$-algebra structure as follows: the sum of two
elements is defined in the obvious way (as functions on $\GG$) while
for the multiplication and the adjoint we use
\begin{gather*}
G_1\delta_x\cdot G_2\delta_y=
G_1(\lambda(x)G_2) \delta_{xy} \\
(G\delta_x)^* =(\lambda(x^{-1})G^*)\delta_{x^{-1}}
\end{gather*}
and extend by linearity.

For any covariant representation $(\pi,\alpha,\mcH)$ of $(\GG,\mcA)$
and, for $\xi=\sum_i G_i\delta_{x_i}$ define
\begin{equation*}
(\pi\ltimes\alpha)(\xi)=\sum_i\alpha(G_i)\pi(x_i)\period
\end{equation*}
Using the covariance relation in Definition~\ref{covariant}, one sees
that $\pi\ltimes\alpha$ defines a $*$-representation of $\mcA[\GG]$.
Define a norm on $\mcA[\GG]$ by
\begin{equation*}
\norm{\xi}=\sup\norm{\pi\ltimes\alpha(\xi)}
\end{equation*}
where the supremum is taken over all covariant representations of
$(\GG, \mcA)$.  The \defn{full crossed product $C^*$-algebra
  $\GG\ltimes\mcA$} is defined as the completion of $\mcA[\GG]$ with
respect to the above norm.

In this paper we are interested in the following cases:
\begin{itemize}
\item $\mcA=\CC$, the complex numbers, with the trivial action of
  $\GG$.  In this case $\GG\ltimes\CC$ is $C^*(\GG)$, the \defn{full
  $C^*$-algebra of $\GG$}.
\item $\mcA=C(K)$ where $K$ is a second countable compact space on
  which $\GG$ acts by homeomorphisms.
\end{itemize}

\begin{remark}
Let $\one_K$ denote the function identically one on the compact space
$K$.
The inclusion $\CC\to C(K)$ defined by $z\to z\one_K$ induces a map
$\psi:C^*(\GG)\to \GG\ltimes C^*(K)$ defined by
\begin{equation*}
\psi(\sum_i c_i\delta_{x_i})=
\sum_ic_i\one_K\delta_{x_i}\period
\end{equation*}
It is trivial to check that this formula gives a $*$-homomorphism
$\CC[\Gamma]\to C(K)[\Gamma]$. To pass to the completions, one needs
$\psi$~to be norm-decreasing, and this follows because any covariant
representation $(\pi,\alpha,\mcH)$ for $\Gamma\ltimes C(K)$ restricts
to a covariant representation $(\pi,\alpha|_{\CC\one_K},\mcH)$ for
$C^*(\Gamma)=\Gamma\ltimes\CC$.

\end{remark}

\begin{definition}\label{wk-cont}
A unitary representation $\pi:\Gamma\to\mcU(\mcH)$ is \defn{weakly
  contained in the regular representation~$\pir$} if for every
$v\in\mcH$ there exists a sequence $v_n\in\ell^2(\Gamma)$ such that
\begin{equation*}
\langle\pi(x)v,v\rangle=\lim_{n\to\infty}
\langle\pir(x)v_n,v_n\rangle\quad
\text{pointwise.}
\end{equation*}
\end{definition}
Based on any $\Gamma$-representation~$\pi$, we define a
$C^*(\Gamma)$-representation, also denoted~$\pi$. Let
$\pi:C^*(\Gamma)\to\mcB(\mcH)$ be the extension of the original~$\pi$
by linearity and continuity.  The function
$\phi(x)=\langle\pi(x)v,v\rangle$ used in the above definition is
known as the \defn{matrix coefficient} associated to~$v$. This also
extends by linearity and continuity to a functional
$\phi:C^*(\Gamma)\to\CC$ given by
$\phi(\xi)=\langle\pi(\xi)v,v\rangle$. If we choose~$v$ with
$\norm{v}=1$, this functional is the \defn{state} corresponding
to~$v$.

Expressing Definition~\ref{wk-cont} using
$C^*(\Gamma)$-representations and states, one finds that a
representation $\pi:C^*(\Gamma)\to\mcB(\mcH)$ is weakly contained in
the regular representation if and only if, for every vector $v\in
\mcH$ with $\norm{v}=1$, the corresponding state $\langle
\pi(\cdot)v,v\rangle$ is a limit, in the weak*-topology, of states
associated with the regular representation, that is of states of the
form $\langle\pir(\cdot)v_n,v_n\rangle$ with~$v_n\in\ell^2(\GG)$.

The proof of the following proposition depends on basic $C^*$-algebra
theory. Only here do we make use of that theory, or of $C^*(\Gamma)$,
or of the fact that~$\cpK$ is a $C^*$-algebra.
\begin{proposition}\label{inclusion-in-boundary} 
Let $\Gamma$~be a discrete countable group 
  acting on a second countable, compact space~$K$. Suppose that the
  unitary representation $\pi:\Gamma\to\mcU(\mcH)$~is weakly contained
  in the regular representation. Then there exists an isometric
  $\Gamma$-inclusion $\iota:\mcH\to\mcH'$ where $\mcH'$ is the
  representation space of a $\cpK$-representation.
\end{proposition}
\begin{proof}
  Fix any~$k_0\in K$ and define a~$C(K)$-action on~$\ell^2(\Gamma)$ by
  \begin{equation*}
    (\pir'(G)f)(x)=G(xk_0)f(x) \period
  \end{equation*}
  Together with $\pir$, this gives a covariant representation as
  per Definition~\ref{covariant} and so gives a representation
  $\pir\ltimes\pir'$  of~$\cpK$ on~$\ell^2(\Gamma)$.

  If the theorem is proved for each summand of a direct sum of
  representations, it is also proved for the sum. Consequently, we may
  assume that $\mcH$ has a unit vector~$v$ cyclic for~$\pi$.  Let
  $\phi$~be the state of~$C^*(\Gamma)$ corresponding to~$v$. By
  hypothesis, $\phi$~is the weak*-limit of a sequence $(\phi_n)_n$,
  where each $\phi_n$ is the state of~$C^*(\Gamma)$ corresponding to a
  vector~$v_n\in\ell^2(\Gamma)$. From the representation
  $\pir\ltimes\pir'$ and~$v_n$ one obtains also a state~$\phi_n'$
  of~$\cpK$. The restriction of~$\phi'_n$ to~$C^*(\Gamma)$ via the map
  $C^*(\Gamma)\to\cpK$ is~$\phi_n$.  The hypotheses on~$\Gamma$
  and~$K$ guarantee that $\Gamma\ltimes C(K)$ is separable and
  unital. So, passing to a subsequence, we may assume that $\phi'_n$
  weakly approaches some positive functional~$\phi'$ on~$\cpK$. The
  restriction of~$\phi'$ to~$C^*(\Gamma)$ will be~$\phi$. Apply the
  Gelfand--Naimark procedure to~$\phi'$ to generate a representation
  $\pi':\cpK\to\mcB(\mcH')$.  The state~$\phi'$ is the state of a
  certain vector $v'\in\mcH'$, cyclic for~$\pi'$. If we restrict
  $\pi'$ to~$C^*(\Gamma)$, the state corresponding to~$v'$ is the
  restriction of~$\phi'$, namely~$\phi$. This was also the state
  corresponding to~$v$ for~$\pi$, so there exists a unique
  $\Gamma$-isometry $\iota:\mcH\to\mcH'$ with $\iota(v)=v'$.
\end{proof}

Proposition \ref{inclusion-in-boundary} proves one implication of
Proposition~\ref{tempered}. The proof of the other implication has
been known for some time and we shall give references:
\begin{itemize}
\item
The action of~$\Gamma$ on~$\bGG$ is \emph{topologically amenable}: see
\cite{Adams} for a general hyperbolic group or the Appendix
of~\cite{K-S1} for the specific case of a free group.
\item \cite[Chapter~X, Theorem~3.8 and 3.15]{T} explains how to
  realize a $\cp$-representation space as~$L^2(\bGG,d\mu)$.
\item
Topological amenability implies that every unitary representation of
$\GG$ that is realized on~$L^2(\bGG,d\mu)$ is weakly contained in the
regular representation. For the precise statement see~\cite{Ku}.
\end{itemize}

%
%

\section{Boundary intertwiners $\iota$, maps~$\mu$, and vectors~$F$}
  \label{imfS}
Throughout this section we will be dealing with a fixed tempered
unitary representation $\pi:\Gamma\to\mcU(\mcH)$.

\subsection{From~$\iota$, to~$\mu$, to~$F$, and back again}
We are basically interested in the \defn{boundary realizations}
of~$\pi$, that is isometric $\Gamma$-maps $\iota:\mcH\to\mcH'$ where
\begin{itemize}
\item $\mcH'$ is the representation space of a $\cp$-representation~$\pi'$,
\item and $\iota(\mcH)$ is cyclic for the action of~$\cp$ on~$\mcH'$.
\end{itemize}
However it is often convenient to drop the condition that $\iota$~be
an isometric inclusion, and to consider \emph{all} $\Gamma$-maps
$\mcH\to\mcH'$ satisfying these two conditions. Such maps are called
\defn{boundary intertwiners} of~$\pi$. As is natural, we call two
boundary intertwiners~$\iota_1:\mcH\to\mcH_1'$
and~$\iota_2:\mcH\to\mcH_2'$ \defn{equivalent} if there exists a
unitary $\cp$-equivalence $J:\mcH'_1\to\mcH'_2$ such that $\iota_2=J\iota_1$,
and we shall write $\iota_1\sim\iota_2$.

\begin{remark} Because of the covariance condition for
  $\cp$-representations, $\pi'(C(\bGG))\iota(\mcH)$ is $\pi'(\Gamma)$
  invariant.  Therefore $\iota(\mcH)$ is cyclic for the action
  of~$\cp$ if and only if it is cyclic for the action of~$C(\bGG)$.
\end{remark}

We are now going to establish the correspondence between the set of
(equivalence classes of) boundary intertwiners of~$\pi$ and two other
sets of objects. To establish a certain parallelism, we will call the
boundary intertwiners~\defn{iota-intertwiners}.
\begin{definition}\label{iota-mu-F} An \defn{iota-intertwiner}
  of~$\pi$ is none other than a boundary intertwiner of~$\pi$. A
  \defn{mu-map} for~$\pi$ is a linear map $\mu:C(\bGG)\to\mcB(\mcH)$
  which takes non-negative functions to positive~semidefinite
  operators (a positive map) and also satisfies
  \begin{equation}\label{muCondE}
    \pi(x) \mu(G) \pi(x)^{-1}=\mu(\lambda(x) G)
    \qquad\text{for $x\in\Gamma$.}
  \end{equation}
  An \defn{Eff-vector} for~$\pi$ is a vector~$F$ of
  positive~semidefinite operators in~$\mcB(\mcH)$, indexed by~$A$,
  and satisfying $\mcT F=F$ where
  \begin{equation}\label{FDefE}
  (\mcT F)_a=\sum_{b\in A \st ab\neq 1} \pi(a)F_b\pi(a)^{-1} \period
    \end{equation}
\end{definition}

We will use several times the following standard elementary lemma.
\begin{lemma}\label{PosIsBdded}
  If $\mu:C(\bGG)\to\mcB(\mcH)$ is a positive map, then there exists a
  constant $C>0$ such that $\norm{\mu(G)}\leq C\norm{G}_\infty$.
\end{lemma}
\begin{proof}
  If $G\geq 0$, then $0\leq G\leq\norm{G}_\infty\one$, hence
  $0\leq\mu(G)\leq\norm{G}_\infty\mu(\one)$, hence
  $\norm{\mu(G)}\leq\norm{\mu(\one)}\norm{G}_\infty$. To treat an
    arbitrary~$G$, write it as the sum of its positive and negative
    real and imaginary parts.
\end{proof}

If an iota-intertwiner $\iota:\mcH\to\mcH'$ is given, we associate to it the
following mu-map:
\begin{equation}\label{mu-defn}
  \mu(G)=\iota^*\pi'(G)\iota \period
\end{equation}
To show that this is a mu-map, one checks positivity, which is
trivial, and covariance
\begin{multline*}
\pi(x)\mu(G)\pi(x)^{-1}  
=\pi(x)\iota^*\pi'(G)\iota\pi(x)^{-1} \\
=\iota^*\pi'(x)\pi'(G)\pi'(x)^{-1}\iota 
=\iota^*\pi'(\lambda(x)G)\iota
=\mu(\lambda(x)G) \period
\end{multline*}

\begin{lemma}\label{Equiv}
Given iota-intertwiners $(\iota_1,\mcH_1')$ and $(\iota_2,\mcH_2')$ of
$\pi$ with associated mu-maps $\mu_1$ and $\mu_2$, then
$\iota_1\sim\iota_2$ if and only if $\mu_1=\mu_2$.
\end{lemma}
\begin{proof}
Recall that $\iota_1\sim\iota_2$ if and only if there exists a unitary 
$\cp$-map $J:\mcH_1'\to \mcH_2'$ such that this diagram commutes:
\begin{equation*}
\begin{tikzcd}[row sep=tiny]
  & \mcH_1'\arrow[dd,"J"] \\
    \mcH\arrow[ru,"\iota_1" pos=0.7] \arrow[rd,"\iota_2"' pos=0.7] \\
  & \mcH'_2
\end{tikzcd}
\end{equation*}
When such a $J$ exists it is obvious that
\begin{equation*}
\mu_2(G) =
\iota_2^*\pi_2'(G)\iota_2 = \iota_1^*J^*\pi_2'(G)J\iota_1 =
\iota^*_1\pi'_1(G)\iota_1 = \mu_1(G) \period
\end{equation*}

Assume now that $\mu_1=\mu_2$. We shall first define $J$ on the dense
subspace of $\mcH_1'$ consisting of finite linear combinations
$\sum_j\pi_1'(G_j)\iota_1(v_j)$ with $G_j\in C(\bGG)$ and $v_j\in
\mcH$ by letting
\begin{equation} 
J\left(\sum_j\pi_1'(G_j)\iota_1(v_j)\right)=
\sum_j\pi_2'(G_j)\iota_2(v_j)
  \period
\end{equation}
Since
\begin{multline*}
\| \sum_j\pi_1'(G_j)\iota_1(v_j)\|^2  
=\sum_{j,k}\langle\pi_1'(G_j)\iota_1(v_j),\pi_1'(G_k)\iota_1(v_k)\rangle \\
=\sum_{j,k}\langle\iota_1^*\pi_1'(\overline{G_k}G_j)\iota_1(v_j),v_k\rangle
=\sum_{j,k}\langle\mu_1(\overline{G_k}G_j)v_j,v_k\rangle \\
=\sum_{j,k}\langle\mu_2(\overline{G_k}G_j)v_j,v_k\rangle
=\| \sum_j\pi_2'(G_j)\iota_2(v_j)\|^2 \comma
\end{multline*}
$J$ is well defined and extends to an isometry from $\mcH_1'$ to
$\mcH_2'$.  Likewise we see that $J^{-1}$ exists and is isometric,
so $J$~is unitary.  It follows from the definition that
$J$~intertwines the two actions of~$C(\bGG)$ on the dense set
$\sum_j\pi_1'(G_j)\iota_1(v_j)$ and hence everywhere.  Finally, to see
that $J$~is a $\GG$-map compute
\begin{multline*}
\pi_2'(x)J\left(\sum_j\pi_1'(G_j)\iota_1(v_j)\right)=
\pi_2'(x)\left(\sum_j\pi_2'(G_j)\iota_2(v_j)\right)\\
=\sum_j\pi_2'(x)\pi_2'(G_j)\pi_2'(x)^{-1}\pi_2'(x)\iota_2(v_j)
=\sum_j\pi_2'(\lambda(x)G_j)\iota_2(\pi(x)v_j)\\
= J\left(\sum_j\pi_1'(\lambda(x)G_j)\iota_1(\pi(x)v_j)\right)
= J \pi_1'(x)\left(\sum_j\pi_1'(G_j)\iota_1(v_j)\right)
 \period
\end{multline*}
\end{proof}

\begin{proposition}\label{mu-gives-iota}
Assume that $\mu:C(\bGG)\to\mcB(\mcH)$ is a mu-map for $\pi$.  Then
there exists an iota-intertwiner $\iota$ such that $\mu$ is the mu-map
associated to $\iota$.
\end{proposition}
\begin{proof}
Let $C(\bGG)\otimes\mcH$ be the algebraic tensor product of $C(\bGG)$ and $\mcH$.
For finite sums $X=\sum_j G_j\otimes v_j$, $Y=\sum_j H_j\otimes w_j$  define
\begin{equation}\label{prod-int-mu}
\langle \sum_j G_j\otimes v_j,\sum_j H_j\otimes w_j\rangle=
 \sum_{j,k}\langle\mu(\overline{H_k} G_j)v_j,w_k\rangle
\end{equation}
where on the right-hand side of~\eqref{prod-int-mu} we use the inner
product of~$\mcH$.

By Stinespring's Theorem \cite{Sti} $\mu$ is \emph{completely
  positive} and it follows that \eqref{prod-int-mu} defines a
semidefinite inner product on $C(\bGG)\otimes\mcH$. Let $\mcH'$ be the
quotient-completion of $C(\bGG)\otimes\mcH$ with respect to this inner
product.

Define
\begin{equation}
\iota:\mcH\to C(\bGG)\otimes\mcH\quad \text{by letting}\quad
\iota(v)=\one\otimes v
  \period
\end{equation}
Let $C(\bGG)$ act on $\sum_j G_j\otimes v_j$ by
\begin{equation}
\pi'(G)\sum_j G_j\otimes v_j=\sum_j G G_j\otimes v_j
 \end{equation}
and check that 
\begin{align*}
&\pi'(GH) =\pi'(G)\pi'(H)\\
&\pi'(G+H)=\pi'(G)+\pi'(H)\\
  &\langle\pi'(G) X,Y\rangle=\langle X,\pi'(\overline{G})Y\rangle
    \period
\end{align*}
In particular, if $G$ is positive one has
\begin{equation}
\pi'(G)=\pi'(\sqrt{G}\sqrt{G})=\pi'(\sqrt{G})(\pi'(\sqrt{G}))^*
\end{equation}
and hence $\pi'(G)$ is a positive operator. 
In order to extend $\pi'(G)$ to all $\mcH'$ we need to show that
$|\langle\pi'(G)X,Y\rangle|\leq C\|G\|_\infty\cdot \|X\|\cdot \|Y\|$
for some constant $C$.
Assume first that $G$ is positive. Since
$0\leq G\leq\norm{G}_\infty\one$, one has
\begin{equation*}
\norm{\pi'(G)}=\sup_{\|X\|\leq1}
\langle\pi'(G)X,X\rangle
\leq\sup_{\|X\|\leq1}
\langle \pi'(\norm{G}_\infty\one) X,X\rangle=\norm{G}_\infty
\comma
\end{equation*}
 and hence $\norm{\pi'(G)}\leq\norm{G}_\infty$.
When $G$~is not positive, divide it up into its positive and negative
real and imaginary parts.

Make $\GG$ act on $C(\bGG)\otimes\mcH$ by 
\begin{equation}
  \pi'(x)\sum_j G_j\otimes v_j=\sum_j \lambda(x)G_j\otimes\pi(x)v_j
    \period
\end{equation}
It is obvious that $\pi'$ defines a group action. It is easy to check
that $\pi'(x)\iota=\iota\pi(x)$.  To see that $\pi'(x)$ is unitary
compute
\begin{align*}
\langle \pi'(x)\sum_j G_j\otimes v_j&,\sum_k H_k\otimes w_k\rangle
=\sum_{j,k}\langle\mu((\lambda(x)G_j)\overline{H_k})\pi(x)v_j,w_k\rangle
  \\
&=\sum_{j,k}\langle\mu(G_j(\lambda(x^{-1})\overline{H_k}))v_j,
  \pi(x)^{-1}w_k\rangle \\
&=\langle \sum_j G_j\otimes v_j, 
  \sum_k\lambda(x^{-1})H_k\otimes\pi(x)^{-1} w_k\rangle \\
&=\langle \sum_j G_j\otimes v_j,\pi'(x)^{-1}\sum_k H_k\otimes
  w_k\rangle
    \period
\end{align*}
Since $\pi'(x)$~is bounded on $C(\bGG)\otimes\mcH$, it extends by
continuity to~$\mcH'$. 

To see that $\pi':\Gamma\to\mcU(\mcH')$ and
$\pi':C(\bGG)\to\mcB(\mcH)$ satisfy~\eqref{bdry-rpn} and give a
$\cp$-representation, compute
\begin{multline*}
\pi'(x)\pi'(G)\pi'(x^{-1}) \sum_j G_j\otimes v_j
=\pi'(x)\sum_j G(\lambda(x^{-1})G_j)\otimes\pi(x)^{-1}v_j\\
=\sum_j(\lambda(x)G) G_j\otimes v_j=\pi'(\lambda(x)G)\sum_j
  G_j\otimes v_j
    \period
\end{multline*}
This completes the construction of the iota-intertwiner. Now we check
that the associated mu-map is the one we wanted:
\begin{equation*}
  \langle \iota^*\pi'(G)\iota v_1,v_2 \rangle
  =\langle \pi'(G)\iota v_1,\iota v_2 \rangle
  =\langle G\otimes v_1, \one\otimes v_2 \rangle
  =\langle \mu(G)v_1,v_2 \rangle \period
\end{equation*}
\end{proof}
Proposition~\ref{mu-gives-iota} is a special case of a considerably
more general fact: see Lemma~3.1 of~\cite{RSW}.

To any given mu-map $\mu: C(\bGG)\to\mcB(\mcH)$, associate an
Eff-vector $F(\mu)=F=(F_a)_{a\in A}$ by letting
\begin{equation}\label{Eff-vector-con-mu}
F_a=\mu(\one_a) \period
\end{equation}
To check that $F$~is an Eff-vector, observe that each~$F_a$ is
positive semi-definite and check that $\mcT F=F$:
\begin{align*}
  (\mcT F)_a &=\sum_{b\in A\st ab\neq 1}
 \pi(a)F_b\pi(a)^{-1}
 =\sum_{b\in A\st ab\neq 1}\mu(\lambda(a)\one_b)\\
 & =\sum_{b\in A\st ab\neq 1}\mu(\one_{ab})= \mu(\one_a)= F_a
 \period
\end{align*}

For mu-maps~$\mu_1$ and $\mu_2$, we say $\mu_1\leq\mu_2$ if
$\mu_1(G)\leq\mu_2(G)$ as operators for every $G\geq 0$. For
Eff-vectors~$F_1$ and~$F_2$, we say that $F_1\leq F_2$ if $(F_1)_a\leq
(F_2)_a$ as operators for each $a\in A$. Indeed we will use this
notation whenever~$F_1$ and~$F_2$ are $|A|$-tuples of operators.
\begin{lemma}\label{Eff-leq-mu}
Assume that $\mu_i:C(\bGG)\to\mcB(\mcH)$ ($i=1,2$) are mu-maps for $\pi$ and let
$F_i=F(\mu_i)$. Then $F_1=F_2$ if and only if $\mu_1=\mu_2$.
Furthermore $F_1\leq F_2$ if and only if $\mu_1\leq\mu_2$.
\end{lemma}
\begin{proof}
Assume that $F_1=F_2$. Take any $a\in A$ and let $x\in\GG$ be such that
$|xa|=|x|+1$.
One deduces successively:
\begin{align*}
              \mu_1(\one_a) &=\mu_2(\one_a) \comma\\
\pi(x)\mu_1(\one_a)\pi(x)^{-1}&=\pi(x)\mu_2(\one_a)\pi(x)^{-1} \comma\\
\mu_1(\lambda(x)\one_a) &=\mu_2(\lambda(x)\one_a) \comma\\
\mu_1(\one_{xa})&=\mu_2(\one_{xa})  \period
\end{align*}
Hence $\mu_1(\one_y)=\mu_2(\one_y)$ for all $y\in\GG$. By linearity,
$\mu_1$ and~$\mu_2$ agree on all locally constant functions. By
continuity, they agree everywhere.  If we replace the equalities in
the above calculation with inequalities, we see that $F_1\leq F_2$
implies $\mu_1(\one_y)\leq\mu_2(\one_y)$ for all $y\in\GG$. Since every
positive function of $C(\bGG)$ can be uniformly approximated by {\it
  positive} linear combinations of $\one_y$ ($y\in\GG$) we may conclude
that $\mu_1\leq\mu_2$.
\end{proof}

\begin{proposition}\label{Eff-gives-mu}
  Let~$F$ be any Eff-vector for~$\pi$. Then there exists
  a unique mu-map $\mu$ for which $F$~is the associated Eff-vector.
\end{proposition}
\begin{proof}
  Uniqueness is the first part of Lemma~\ref{Eff-leq-mu}
  above. Given~$F$, we must
  construct~$\mu$. By~\eqref{Eff-vector-con-mu} we must have
  $\mu(\one_a)=F_a$, and so by~\eqref{muCondE}
  \begin{equation}\label{Eff-gives-mu-eq}
    \mu(\one_{xa})=\mu(\lambda(x)\one_a)=\pi(x)F_a\pi(x)^{-1}
    \qquad\text{when $|xa|=|x|+1$.}
  \end{equation}
  Working with~\eqref{Eff-gives-mu-eq} we start by defining~$\mu$ on
  functions~$G(\omega)$ which depend only on the first $n$~letters
  of~$\omega$.
  \begin{equation*}
    \mu(G)=\mu\bigl(\sum_{xa \st |xa|=|x|+1=n} G_{xa}\one_{xa}\bigr)
      =\sum_{xa \st |xa|=|x|+1=n} G_{xa}\pi(x)F_a\pi(x)^{-1} \period
  \end{equation*}
  This same function $G(\omega)$ might also be considered as depending
  on the first $n+1$~letters of~$\omega$. This way of looking at~$G$
  gives a different formula for $\mu(G)$:
  \begin{equation*}
    \mu(G)=\sum_{xab \st |xab|=|x|+2=n+1} G_{xa}\pi(xa)F_b\pi(xa)^{-1}
  \end{equation*}
  and our definition is consistent only if the two answers agree.
  They do agree due to the condition $F=\mcT F$ with~$\mcT$ as in
  equation~\eqref{FDefE}.

  Now~$\mu(G)$ is defined for all functions~$G(\omega)$ which depend
  on only finitely many letters of~$\omega$. Let $C^\infty(\bGG)$
  denote the subalgebra of all such functions.  $\mu$ is a positive
  map on~$C^\infty(\bGG)$ because $F_a\geq 0$ for all~$a\in A$. To
  extend~$\mu$ to all of $C(\bGG)$ by continuity we need the
  inequality $\norm{\mu(G)}\leq C\norm{G}_\infty$ which follows as
  in the proof of Lemma~\ref{PosIsBdded}.


  Now we check for covariance:
  $\pi(x)\mu(G)\pi(x)^{-1}=\mu(\lambda(x)G)$.  By continuity, it is
  enough to check this when $G\in C^\infty(\bGG)$.  By linearity, it
  is enough to consider $G=\one_{ya}$ where $|ya|=|y|+1>|x|$. Then
  \begin{align*}
    \pi(x)\mu(\one_{ya})\pi(x)^{-1}
    &=\pi(x)\pi(y)F_a\pi(y)^{-1}\pi(x)^{-1} \\
    &=\pi(xy)F_a\pi(xy)^{-1} 
    =\mu(\one_{xya})
    =\mu(\lambda(x)\one_{ya}) \period
  \end{align*}
\end{proof}

\subsection{Perfect realizations seen in terms of~$\iota$, $\mu$,
  and~$F$}

\begin{lemma}\label{perfect-by-mu}
Suppose that $\iota$ is an iota-intertwiner for $\pi$ with associated
mu-map $\mu$. Then
\begin{itemize}
\item $\iota$ is a boundary realization if and only of $\mu(\one)=\ID$.
\item $\iota$ is a perfect boundary realization if and only if $\mu$
is an algebra homomorphism.
\end{itemize}
\end{lemma}
\begin{proof}
Observe that $\mu(\one)=\iota^*\pi'(\one)\iota=\iota^*\iota$.  Hence
$\iota$ is an isometry if and only if $\mu(\one)=\ID$.  Consider the
second statement. If $\iota$ is perfect one has
$\iota^*\iota=\iota\iota^*=\ID$. Consequently
$\mu(G_1G_2)=\iota^*\pi'(G_1G_2)\iota=
\iota^*\pi'(G_1)\iota\iota^*\pi'(G_2)\iota=\mu(G_1)\mu(G_2)$.
The converse is more delicate.  When $\mu$ is an algebra homomorphism
we can make $\mcH$ itself into a $\cp$-representation space by
defining $\tilde\pi'(x)=\pi(x)$ and $\tilde\pi'(G)=\mu(G)$ for
$x\in\Gamma$ and $G\in C(\bGG)$. Letting $\tilde\mcH'=\mcH$ and
$\tilde\iota=\ID$, we get a perfect boundary realization
of~$\mcH$. The mu-map corresponding to this realization is
$G\mapsto \tilde\iota^*\tilde\pi'(G)\tilde\iota=\mu(G)$. Since~$\iota$
and $\tilde\iota$ give the same mu-map, Lemma~\ref{Equiv} says they
are equivalent, so $\iota$~is also a perfect realization.
\end{proof}

\begin{lemma}\label{F-perfect}
Suppose that $\iota$ is an iota-intertwiner for $\pi$ with associated
Eff-vector $(F_a)$. Then
\begin{itemize}
\item $\iota$ is a boundary realization if and only of $\sum_{a\in A}F_a=\ID$.
\item $\iota$ is a perfect boundary realization if and only if in addition
$F_a^2=F_a$ for each $a\in A$.
\end{itemize}
\end{lemma}
\noindent We need this well-known elementary lemma:
\begin{lemma}\label{proj-sum}
Suppose that $(P_j)_{1\leq j\leq n}$ are self-adjoint projections such that
$\sum_{j=1}^n P_j=\ID$. Then the $P_j$ are mutually orthogonal. 
\end{lemma}
\begin{proof}
Assume that $v=P_kv$. One has
\begin{align*}
\langle v,v\rangle &=\langle \sum_{j=1}^n P_j P_kv,v\rangle=
\langle P_k^2v,v\rangle +\sum_{j\neq k}\langle P_jP_kv,v\rangle\\
  &=\langle v,v\rangle +\sum_{j\neq k}\langle P_jP_kv,v\rangle
  \period
\end{align*}
Since $\langle P_jP_kv,v\rangle=\langle P_jP_kv,P_kv\rangle\geq0$ all
such summands in the above sum must be zero. Hence
$P_kP_jP_k=P_kP_j(P_kP_j)^*=0$ for $j\neq k$.
\end{proof}
\begin{proof}[Proof of Lemma~\ref{F-perfect}]
$\iota$ is a boundary realization if and only if $\mu(\one)=\ID$. Since
 $\one=\sum_{a\in A}\one_a$ and $F_a=\mu(\one_a)$, 
 $\mu(\one)=\ID$ if and only if
$\sum_{a\in A}F_a=\ID$.

If $\iota$ is a perfect realization, then $\mu$~is an algebra
homomorphism, and so each $F_a=\mu(\one_a)$ satisfies
$F_a^2=F_a$. Vice versa, suppose that each~$F_a$ is a projection,
necessarily orthogonal. Then for any $xa\in\Gamma$ with $|xa|=|x|+1$,
one has that $\mu(\one_{xa})= \pi(x)\mu(\one_a)\pi(x)^{-1}=
\pi(x)F_a\pi(x)^{-1}$ is also an orthogonal projection. For any~$n\geq
1$, one has $\ID=\mu(\one)=\sum_{|xa|=|x|+1=n}
\mu(\one_{xa})$ and so one may apply Lemma~\ref{proj-sum} and deduce
that the projections in this sum are all mutually orthogonal, and the
product of any two of them is zero. It follows easily that $\mu$~is an
algebra homomorphism when restricted to the subalgebra
$\sum_{|xa|=|x|+1=n}\CC\,\one_{xa}$. This holds
for any~$n$, so~$\mu$ is an algebra homomorphism on $C^\infty(\bGG)$,
hence on all of~$C(\bGG)$.
\end{proof}

\subsection{Direct sums of boundary intertwiners}\label{dir-sums}
Let~$\mu_1$ and~$\mu_2$ be two mu-maps, and let $(\iota_1,\mcH_1)$ and
$(\iota_2,\mcH_2)$ be the associated boundary intertwiners. What is
the boundary intertwiner associated to $\mu_1+\mu_2$?  Consider
$(\iota_1\oplus\iota_2,\mcH_1'\oplus\mcH_2')$. It is trivial to check
that this has $\mu_1+\mu_2$ as its mu-map. However it may not satisfy
the condition that $(\iota_1\oplus\iota_2)(\mcH)$ is cyclic in
$\mcH_1'\oplus\mcH_2'$. So let $\mcH_S'$~be the closure of
$(\pi'_1\oplus\pi'_2)(C(\bGG))(\iota_1\oplus\iota_2)(\mcH)$, and
define $\iota_S$ by the following diagram:
\begin{equation*}
\begin{tikzcd}[row sep=tiny]
  & \mcH_S'\arrow[dd] \\
  \mcH\arrow[ru,"\iota_S" pos=0.3]
    \arrow[rd,"\iota_1\oplus\iota_2"' pos=0.7] \\
  & \mcH'_1\oplus\mcH'_2
\end{tikzcd}
\end{equation*}
The boundary intertwiner we are looking for is~$(\iota_S,\mcH'_S)$,
which we will denote by $(\iota_1,\mcH_1')\oplus(\iota_2,\mcH_2')$
and call the \defn{direct sum} of the two boundary intertwiners.

\begin{lemma}\label{sum-two-irrin}
Let $(\iota_j,\mcH'_j)_{j=1,2}$ be boundary intertwiners
mapping~$\mcH$ to \emph{irreducible and inequivalent} $\cp$-spaces
$(\mcH'_j)_{j=1,2}$. Then their direct sum is just
$(\iota_1\oplus\iota_2,\mcH'_1\oplus\mcH'_2)$.
\end{lemma}
\begin{proof}
  By construction,
  $\mcH'_S$ is a $\cp$-subspace of
  $\mcH_1'\oplus\mcH_2'$. Because~$\mcH_1'$ and~$\mcH_2'$ are
  irreducible and inequivalent, the only possibilities for $\mcH'_S$
  are $0$, $\mcH'_1$, $\mcH'_2$, and $\mcH'_1\oplus\mcH'_2$. Since
  $\iota_1(\mcH)\oplus\iota_2(\mcH)\subseteq\mcH'_S$ we see that
  only $\mcH'_S=\mcH'_1\oplus\mcH'_2$ is possible.
\end{proof}


\subsection{Scalar multiples of boundary intertwiners}\label{scalar-mults}
Let $\mu$~be a mu-map and let $(\iota,\mcH')$ be the associated
boundary intertwiner. For $t>0$, $t\mu$~is another mu-map, whose
associated boundary intertwiner is $(\sqrt{t}\,\iota,\mcH')$. For
$t=0$, we have $t\mu=0$, and the associated boundary intertwiner
is the zero map to the zero $\cp$-space.

\subsection{What does $\mu\leq\mu_1$ mean?}

\begin{proposition}\label{sub-rep}
Assume that $\mu$ and $\mu_1$ are mu-maps for $\pi$ with associated
boundary intertwiners $(\iota,\mcH')$ and $(\iota_1,\mcH'_1)$. If
$\mu\leq\mu_1$ then there exists a $\cp$-map $\phi:\mcH'_1\to \mcH'$
so that $\phi\iota_1=\iota$. The image of~$\phi$ is dense and
$\norm{\phi}\leq 1$.
\end{proposition}

\begin{proof}
By Lemma~\ref{Equiv}, two intertwiners corresponding to the same
mu-map are equivalent. So we may assume that $(\iota_1,\mcH'_1)$ and,
$(\iota,\mcH')$ are as constructed in the proof of
Proposition~\ref{mu-gives-iota} starting from $\mu_1$ and $\mu$
respectively.  Let $X=\sum_j G_j\otimes v_j$. Using the definition of
the norms one has
\begin{equation}\label{sub-rep-eq}
\norm{X}^2_{\mcH'} =
\sum_{j,k}\langle\mu(\overline{G_k}G_j)v_j,v_k\rangle\leq
\sum_{j,k}\langle
{\mu_1}(\overline{G_k}G_j)v_j,v_k\rangle=
\norm{X}^2_{{\mcH_1}'}
\end{equation}
where the inequality holds because, according to Stinespring's
Theorem, $\mu_1-\mu$ is completely positive.  Hence the identity map
extends to a continuous $\cp$-map $\phi:{\mcH_1}'\to\mcH'$. Clearly
the image is dense, $\norm{\phi}\leq 1$ and
\begin{equation*}
  \phi\iota_1(v)=\phi(\one\otimes v)=\one\otimes v=\iota(v)\in\mcH'
  \period
\end{equation*}
\end{proof}

\begin{corollary}\label{sub-rep-irr}
Let $\mu$, $\mu_1$, $(\iota,\mcH')$, $(\iota_1,\mcH'_1)$ be as in
Proposition~\ref{sub-rep}. If $\mcH'_1$ is an irreducible $\cp$-representation, then $\mu=t\mu_1$ for some nonnegative
constant~$t\leq 1$.
\end{corollary}
\begin{proof}
Let $\phi$ be as in \ref{sub-rep} and set
$T=\phi^*\phi:\mcH_1'\to\mcH_1'$. Then $T$~intertwines $\pi'_1$ to
itself, hence $T=t\ID$ for some~$t\geq 0$.  Then $t\leq 1$ follows
from $\norm{\phi}\leq 1$.
The following diagram commutes:
\begin{equation*}
\begin{tikzcd}[row sep=tiny]
  & \mcH'_1\arrow[dd,"\phi" pos=0.4] \\
    \mcH\arrow[ru,"\iota_1" pos=0.7] \arrow[rd,"\iota"' pos=0.65] \\
  & \mcH'
\end{tikzcd}
\end{equation*}
Using $\iota=\phi\iota_1$ and $\pi'(G)\phi=\phi\pi_1'(G)$ we have
\begin{equation*}
\mu(G)=\iota^*\pi'(G)\iota=\iota_1^*\phi^*\pi'(G)\phi\iota_1=
\iota_1^*\phi^*\phi\pi_1'(G)\iota_1=t\mu_1(G)
  \period
\end{equation*}
\end{proof}

\begin{corollary}\label{SubMuP}
  Let $(\iota_j)_{j=1,2}$ be boundary intertwiners mapping~$\mcH$ to
  \emph{irreducible and inequivalent} $\cp$-spaces
  $(\mcH'_j)_j$.  Let $(\mu_j)_{j=1,2}$ be the corresponding mu-maps.
  If $\mu$~is some other mu-map satisfying
  $\mu\leq C(\mu_1+\mu_2)$ for some~$C>0$, then
  $\mu=t_1\mu_1+t_2\mu_2$ for some pair $(t_1,t_2)$ of nonnegative
  coefficients.
\end{corollary}
\begin{proof}
By scaling we may assume that $C=1$.  Let $(\iota,\mcH')$~be the
boundary intertwiner corresponding to~$\mu$. According to
Lemma~\ref{sum-two-irrin} the boundary intertwiner associated to
$\mu_1+\mu_2$ is
$(\iota_1\oplus\iota_2,\mcH'_1\oplus\mcH'_2)$. Let
$\phi:\mcH'_1\oplus\mcH'_2\to\mcH'$ be as in
Proposition~\ref{sub-rep}. We have
\begin{equation*}
\begin{tikzcd}[row sep=tiny]
  & \mcH'_1\oplus\mcH'_2\arrow[dd,"\phi" pos=0.4] \\
  \mcH\arrow[ru,"\iota_1\oplus\iota_2" pos=0.7]
    \arrow[rd,"\iota"' pos=0.3] \\
  & \mcH'
\end{tikzcd}
\end{equation*}
Since~$\mcH_1'$ and $\mcH_2'$ are irreducible and inequivalent and
since $\phi^*\phi$ is a $\cp$-intertwiner, it must be given by a
block matrix of the form
\begin{equation*}
\phi^*\phi=
\begin{pmat}
t_1 \ID_{\mcH_1'}& 0\\  
0 & t_2\ID_{\mcH_2'}
\end{pmat}
\qquad\text{with $t_1,t_2\geq 0$.}
\end{equation*}

\begin{align*}
  \mu(G)&=\iota^*\pi'(G)\iota
  =(\iota_1\oplus\iota_2)^*\phi^*\pi'(G)\phi(\iota_1\oplus\iota_2) \\
  &=
  (\iota_1\oplus\iota_2)^*\phi^*\phi
    (\pi_1'(G)\oplus\pi_2'(G))(\iota_1\oplus\iota_2) \\
  &=\iota_1^*(t_1\ID_{\mcH_1'})\pi_1'(G)\iota_1
    +\iota_2^*(t_2\ID_{\mcH_2'})\pi_2'(G)\iota_2
  =t_1\mu_1(G)+t_2\mu_2(G)
  \period
\end{align*}

\end{proof}

\section{The trace inner-product}\label{inner-product}
\subsection {$\TR(T_1,T_2)$ for~$T_1$ and~$T_2$
  positive~semidefinite}
We denote by $\mcB^+(\mcH)\subseteq\mcB(\mcH)$ the subset of
positive~semidefinite operators.  For $T\in\mcB^+(\mcH)$ we recall the
definition of the \defn{trace}:
\begin{equation*}
\tr(T)
=\sum_{i=1}^\infty\langle T e_i,e_i\rangle
\end{equation*}
for some fixed orthonormal basis~$\{e_i\}_{i=0}^\infty$.  Let
$S,T\in\positivi$.  It is well known (see for example \cite{Dix2}
Section 1.6.6) that
\begin{itemize}
\item
$\tr(T)\in[0,+\infty]$.
\item 
$\tr(\alpha T+\beta S)=
\alpha\tr(T)+\beta\tr(S)$ for $\alpha,\beta\in
\RR^+$.
\item
$\tr(TT^*)=\tr(T^*T)=
\norm{T}^2_{HS}$ where $\norm{\cdot}_{HS}$ denotes the Hilbert--Schmidt norm.
\item
If $U\in\mcB(H)$ is unitary, then
$\tr(UTU^{-1})=\tr(T)$.
\item
$\tr(T)$ is independent of the choice of basis.
\end{itemize}

\begin{definition}
For $S,T\in\positivi$ we define
\begin{equation}\label{TR}
\TR(S,T)=\tr(\sqrt{S}T\sqrt{S}) \period
\end{equation}
\end{definition}
The following properties are easily deduced from the
above-mentioned properties of the trace:
\begin{itemize}
\item
$\TR(S,T)\in[0,+\infty]$.
\item 
$\TR(S,T)=\norm{\sqrt{S}\sqrt{T}}^2_{HS}$.
\item
  $\TR(S,T)=\TR(T,S)$.
\item $\TR(S,T)$ is bilinear.  
\item
If $T,S\in\mcB^+(H)$ and $\TR(S,T)=0$, then
$ST=TS=0$.
\item If $U$ is unitary $\TR(US\inv U,UT\inv U)=\TR(S,T)$.
\item
If $T\leq A$ and $S\leq B$ then $\TR(T,S)\leq\TR(A,B)$.
\end{itemize}
\begin{proof}
The first statement is obvious. Let us turn to the others: 
\begin{align*}
 \TR(S,T)&=\tr(\sqrt S T \sqrt S)
  =\tr(\sqrt S \sqrt T\sqrt T \sqrt S)\\
  &=\tr((\sqrt S \sqrt T)(\sqrt S \sqrt T)^*)
  =\norm{\sqrt{S}\sqrt{T}}^2_{HS} \\
  &=\tr((\sqrt S\sqrt T)^*(\sqrt S\sqrt T))
  =\TR(T,S)
  \period
\end{align*}
Bilinearity is now obvious.
  Suppose that $\TR(S,T)=0$. By the second statement
  $\sqrt{S}\sqrt{T}=0$. Multiply on the left by $\sqrt{S}$ and
  on the right by $\sqrt{T}$ to get~$ST=0$.
The next to the last  statement follows from the fact that 
$\sqrt{US\inv U}=U\sqrt{S}\inv U$ when $U$ is unitary.
  Finally, assume that~$T\leq A$ and~$S\leq B$. Choose an
  orthonormal basis $e_i$ and compute
  \begin{align*}
  \TR(T,S)&=\tr(\sqrt T S \sqrt T)=
    \sum_i\langle S\sqrt T  e_i,\sqrt T e_i\rangle\\
    &\leq\sum_i\langle B\sqrt T  e_i,\sqrt T e_i\rangle=\TR(T,B)
    \period
  \end{align*}
  Use now that $\TR(T,B)=\TR(B,T)$ to get 
  $\TR(T,S)\leq\TR(B,A)=\TR(A,B)$.
  \end{proof}

We will use the following versions of Fatou's~Lemma and of the Bounded
Convergence Theorem. They are immediate consequences of the usual
Fatou's~Lemma and Bounded Convergence Theorem for counting measure and
we prove them together.
\begin{proposition}\label{FatouP}
  Suppose that $T$~is positive~semidefinite and that the sequence
  $(T_j)_{j\geq 0}$ is made up of positive~semidefinite operators, is
  increasing, and has a weak limit $T_\infty$. Then
  \begin{equation*}
    \TR(T_\infty,T) \leq \liminf_{j\to\infty} \TR(T_j,T)
     \period
  \end{equation*}
\end{proposition}
\begin{proposition}\label{BdConvP} Suppose that~$T$ and $T_B$~are
  positive~semidefinite and satisfy $\TR(T_B,T)<\infty$. Suppose that
  the sequence $(T_j)_{j\geq 0}$ is made up of positive~semidefinite
  operators, all bounded above by~$T_B$, and has a weak limit
  $T_\infty$. Then
  \begin{equation*}
    \TR(T_\infty,T) = \lim_{j\to\infty} \TR(T_j,T)
     \period
  \end{equation*}
\end{proposition}
\begin{proof}
    Choose an orthonormal basis $\{ e_k\}$ for $\mcH$ and consider
  the quantities
  $\langle T_j \sqrt T e_k,\sqrt Te_k\rangle=\varphi_j(k)$. Since
  $(T_j)_{j\geq0}$ has a weak limit we know that $\varphi_j$ is pointwise
  convergent for each $k$.

  To get Proposition~\ref{FatouP} apply Fatou's Lemma to $\varphi_j$
  with respect to counting measure.  To get Proposition~\ref{BdConvP}
  observe that the $\varphi_j$ are all bounded above by
  $\varphi_B(k)=\langle T_B\sqrt T e_k,\sqrt T e_k\rangle$ and apply
  the Bounded Convergence Theorem.
\end{proof}
\subsection {$(F_1,F_2)$ for Eff-vectors~$F_1$ and~$F_2$}
\begin{definition}
  If $F=(F_a)_{a\in A}$ and $\tilde F=(\tilde{F_a})_{a\in A}$ are $|A|$-tuples of
  positive semidefinite operators, define
  \begin{equation}\label{()}
    (F,\tilde F)=\sum_{a\neq b}\TR(F_a,\tilde{F_b})
      \period
    \end{equation}
  \end{definition} 

Given this definition, the following Corollaries follow immediately
from Propositions~\ref{FatouP} and~\ref{BdConvP}.
\begin{corollary}\label{FatouC}
  Suppose that $F$~is an $|A|$-tuple of positive~semidefinite
  operators, that the sequence $(F_j)_{j\geq 0}$ is made up of
  similar tuples, that the sequence is increasing componentwise, and
  has a componentwise weak limit $F_\infty$. Then
  \begin{equation*}
    (F_\infty,F) \leq \liminf_{j\to\infty} (F_j,F)
     \period
  \end{equation*}
\end{corollary}

\begin{corollary}\label{BdConvC} Suppose that~$F$ and $F_B$~are
  $|A|$-tuples of positive~semidefinite operators satisfying
  $(F_B,F)<\infty$. Suppose that the sequence $(F_j)_{j\geq 0}$ is
  likewise made up of $|A|$-tuples of positive~semidefinite operators,
  all bounded above componentwise by~$F_B$, and suppose that the
  sequence has a componentwise weak limit $F_\infty$. Then
  \begin{equation*}
    (F_\infty,F) = \lim_{j\to\infty} (F_j,F)
     \period
  \end{equation*}
\end{corollary}

\begin{proposition}\label{(F,F)=0}
  Let $(\iota,\mcH')$ be a boundary realization with associated
  mu-map~$\mu$ and associated Eff-vector~$F=(F_a)$ as per
  equations~\eqref{mu-defn} and~\eqref{Eff-vector-con-mu}.  Then
  $(\iota,\mcH')$ is perfect if and only if $(F,F)=0$.
\end{proposition}
\begin{proof}
  If $(\iota,\mcH')$ is perfect then $\mu$ is an algebra homomorphism.
  In particular one has
  $F_aF_b=\mu(\one_a)\mu(\one_b)=\mu(\one_a\one_b)=\mu(0)=0$ if $a\neq
  b$ and hence $(F,F)=0$.  Conversely, assume that $(F,F)=\sum_{a\neq
    b}\TR(F_a,F_b)=0$.  Since each $F_a$ is positive semidefinite one
  has $F_aF_b=0=F_bF_a$ when $a\neq b$.  Moreover, since $\iota$ is an
  isometry, one has $\ID=\mu(\one)=\sum_{a\in A}\mu(\one_a)$. Multiply
  both sides by $F_b$ :
 \begin{equation*}
   F_b=\sum_{a\in A}F_bF_a=F_b^2
   \period
 \end{equation*}
 Now apply Lemma~\ref{F-perfect}.
\end{proof}

\begin{proposition}
  Let $(\iota,\mcH')$, $(\tilde\iota,\tilde{\mcH}')$ be two boundary
  realizations with associated mu-maps~$\mu$ and~$\tilde\mu$ and
  associated Eff-vectors~$F=(F_a)$ and~$\tilde F=(\tilde F_a)$.  Then
  $(\iota,\mcH')$ and $(\tilde\iota,\tilde{\mcH}')$ are perfect and
  equivalent if and only if $(F,\tilde F)=0$.
  \end{proposition}
\begin{proof}
  If $(\iota,\mcH')$ and $(\tilde\iota,\tilde{\mcH}')$ are equivalent
  the corresponding Eff-vectors $F$ and $\tilde F$ are equal, so that
  the statement follows from Proposition~\ref{(F,F)=0}.
  Conversely, assume that $(F,\tilde F)=0$ or, equivalently, that
  $F_a\tilde F_b=\tilde F_b F_a=0$ for all $a\neq b$.  Since both
  $\mu$ and $\tilde\mu$ are isometries one has
  \begin{equation}\label{intermezzo}
    \sum_{a\in A} F_a=\sum_{a\in A}\tilde F_a=\ID
    \period
  \end{equation}
  Fix $b\in A$ and multiply the left-hand side of \eqref{intermezzo} by
  $\tilde F_b$ and the right-hand side by $F_b$ to get
  \begin{align*}
    \tilde F_b&=\sum_{a\in A}\tilde F_b F_a=\tilde F_b F_b
      \comma &
    F_b&=\sum_{a\in A}\tilde F_a  F_b=\tilde F_b F_b
    \comma
  \end{align*}
  and conclude that the two realizations are equivalent by
  Lemmas~\ref{Equiv} and~\ref{Eff-leq-mu} and perfect by
  Proposition~\ref{(F,F)=0}.
  \end{proof}
\begin{definition}
   Let $(\iota,\mcH')$ be a boundary intertwiner with associated
   mu-map $\mu$ and associated Eff-vector $F=(F_a)$.  We say that
   $(\iota,\mcH')$ satisfies the \defn{finite trace condition} or
   briefly \eqref{FTC} if
\begin{equation}
  \label{FTC}
  \tag{FTC} (F,F) < \infty \period
\end{equation}
\end{definition}
\begin{remark}
  Since $(F,F)=0$ for any perfect boundary realization, \eqref{FTC} is of
  interest mostly for imperfect boundary realizations.  When
  $(\iota,\mcH)$ is the direct sum of two perfect boundary
  realizations, $(\iota_1,\mcH'_1)$ and $(\iota_2,\mcH'_2)$, the
  corresponding Eff-vector $F$ is the sum of the Eff-vectors
  corresponding to $\iota_1$ and $\iota_2$ and the \eqref{FTC} for
  $F$ becomes
  \begin{equation*}
    (F,F)=2(F_1,F_2)<\infty
   \period 
   \end{equation*}
\end{remark}

\begin{remark}
  The following straightforward property of $(\cdot,\cdot)$ is
  crucial in the next section. Here $a,b,c,d\in A$.
\begin{equation}\label{(TF,F)}
  \begin{aligned}
    (\mcT F,\tilde F)&=\sum_{a\neq b}
      \TR((\mcT F)_a,\tilde{F_b}) \\
     &=\sum_{a\neq b}\sum_{c\neq a^{-1}}\TR(\pi(a)F_c\pi(a)^{-1}
    ,\tilde{F_b})\\
    &= \sum_{\;a\neq b}\sum_{c\neq a^{-1}}\TR(F_c,
    \pi(a)^{-1}\tilde{F_b}\pi(a))\\
    &= \sum_{b\neq d^{-1}}\sum_{c\neq d}\TR(F_c,
    \pi(d)\tilde{F_b}\pi(d)^{-1})
    =(F,\mcT\tilde F)
    \period
  \end{aligned}
  \end{equation}
  
\end{remark}  

\section{A weak limit}\label{weak-limit}

Recall that when~$F_1$ and~$F_2$ are $|A|$-tuples of operators, we
write $F_1\leq F_2$ to mean $(F_1)_a\leq (F_2)_a$ for every $a\in A$.
\begin{proposition}\label{limP} Let $F$~be an Eff-vector. Let~$F_0$ be
  any $|A|$-tuple of positive semidefinite operators satisfying
  $(\mcT^N F_0)_a\leq CF_a$ for some fixed integer~$N\geq 0$, some
  fixed $C>0$, and all~$a\in A$. Then there exists a sequence
  $(\epsilon_j)_j\to 0+$ such that the componentwise weak limit
  \begin{equation*}
    F_L=\wklim_{j\to\infty} \epsilon_j\sum_{n\geq 0} e^{-\epsilon_j n} \mcT^n
       F_0
  \end{equation*}
  exists. Moreover
  \begin{itemize}
  \item $F_L\leq CF$;
  \item $\mcT F_L=F_L$;
  \item if $F_1$~is any Eff-vector satisfying~$(F,F_1)<+\infty$,
    then $(F_L,F_1)=(F_0,F_1)$.
  \end{itemize}
  One can choose $(\epsilon_j)_j$ to be a subsequence of any given
  sequence decreasing to~$0$.
\end{proposition}
\begin{proof}
  From equation~\eqref{FDefE}, which defines~$\mcT$, it follows that
  $\mcT F'\leq \mcT F''$ componentwise whenever $F'\leq F''$
  componentwise.  Since $\mcT F=F$, our hypotheses imply
  \begin{equation*}
    (\mcT^{n} F_0)_a\leq C F_a\quad\text{for all $n\geq N$}
  \end{equation*}
  hence
  \begin{equation}\label{TF1}
    (\epsilon\sum_{n\geq N}e^{-\epsilon n}\mcT^{n} F_0)_a
    \leq \frac{\epsilon}{1-e^{-\epsilon}}C F_a
    \period
  \end{equation}
  From~\eqref{TF1} deduce first that the series
  $\sum_{n\geq0}e^{-\epsilon n}\mcT^nF_0$ converges componentwise in
  the norm topology, and then that the quantities
  \begin{equation*}
    \norm{ (\epsilon\sum_{n\geq 0} e^{-\epsilon n} \mcT^nF_0)_a}
  \end{equation*}
  are uniformly bounded for $0<\epsilon\leq1$.  Since the unit ball of
  $\mcB(\mcH)$ is compact and metrizable in the weak operator
  topology, we conclude that there exists a sequence
  $\epsilon_j\to 0+$ such that
\begin{equation*}
  \wklim_{j\to\infty}\epsilon_j\sum_{n\geq 0} e^{-\epsilon_j n} \mcT^nF_0=F_L      
  \end{equation*}
exists componentwise. One gets $F_L\leq CF$ from~\eqref{TF1}.

From the definition of~$\mcT$ it follows easily that~$\mcT$ commutes
with componentwise weak limits. Thus
\begin{align*}
  \mcT  F_L&=\mcT(\wklim_{j\to\infty}\epsilon_j\sum_{n\geq 0}
    e^{-\epsilon_j n} \mcT^nF_0) 
  =\wklim_{j\to\infty}\mcT\epsilon_j\sum_{n\geq 0}
    e^{-\epsilon_j n} \mcT^nF_0 \\
  &=\wklim_{j\to\infty}\epsilon_j\sum_{n\geq 0}
    e^{-\epsilon_j n} \mcT^{n+1}F_0
  =\wklim_{j\to\infty}e^{\epsilon_j}\epsilon_j\sum_{n\geq 1}
    e^{-\epsilon_j n} \mcT^n F_0=F_L
  \period  
\end{align*}
Finally, assume that $F_1$ is another Eff-vector and that
$(F,F_1)<\infty$.  Use~\eqref{TF1} for boundedness and apply
Corollary~\ref{BdConvC}.
  \begin{align*}
    (F_L,F_1)&=(
    \lim_{j\to\infty}\epsilon_j\sum_{n\geq N} e^{-\epsilon_j n} \mcT^nF_0,F_1)
    =\lim_{j\to\infty}(\epsilon_j\sum_{n\geq N} e^{-\epsilon_j n} \mcT^nF_0,F_1)\\
    &=\lim_{j\to\infty}\epsilon_j\sum_{n\geq N} e^{-\epsilon_j n} (\mcT^nF_0,F_1)
    =\lim_{j\to\infty}\epsilon_j\sum_{n\geq N} e^{-\epsilon_j n}(F_0, \mcT^nF_1)\\
    &=\lim_{j\to\infty}\epsilon_j\sum_{n\geq N} e^{-\epsilon_j n}(F_0,F_1)
    =(F_0,F_1)
    \end{align*}
since $\mcT F_1=F_1$.
\end{proof}

\section{Good vectors}\label{good-vectors}
Throughout this section we consider a fixed 
representation~$\pi$ of~$\Gamma$ on~$\mcH$ and a fixed boundary
realization $\iota:\mcH\to\mcH'$ of~$\pi$. The concepts of \defn{good
  vector} and \defn{special good vector} are relative to this~$\pi$
and this~$\iota$. As usual, let
\begin{align*}
\mu(G)&=\iota^*\pi'(G)\iota &
F_a&=\iota^*\pi'(\one_a)\iota
\end{align*}
be the corresponding mu-map and Eff-vector.


For~$v_1$ and~$v_2\in\mcH$, recall that $v_1\otimes\bar v_2$ stands
for the rank~one operator given by
\begin{equation*}
(v_1\otimes\bar v_2)(v)=\langle {v},{v_2}\rangle v_1
\end{equation*}
and that 
\begin{equation*}
\pi(x)(v_1\otimes\bar v_2)\pi(x)^{-1}=
  (\pi(x)v_1)\otimes\overline{(\pi(x)v_2)}
    \period
\end{equation*}

\begin{definition} Say that a vector $v\in\mcH$ is a \defn{good
  vector} with respect to~$\iota$ if there exist $C>0$, $N\geq 0$ so
  that
  \begin{equation*}
  \mcT^N E \leq CF
  \end{equation*}
  where $E_a=v\otimes\bar v$ for every $a\in A$.
  Say that $v$ is a \defn{special good vector} with respect to~$\iota$
  if for some $z\in\Gamma$ and some $C>0$ we have
\begin{align}\label{eq:sgv}
  v\otimes\bar v&\leq C\mu(\one_z)
  &
  v\otimes\bar v&\leq C\mu(\one-\one_z)
  \period
\end{align}
\end{definition}

\begin{remark}\label{potenza-good-vector} If $v$ is a good vector, arguing as in the proof of
 Proposition~\ref{limP}, we see that
\begin{equation*}
\mcT^n E\leq C F\qquad\text{for all $n\geq N$}\period
\end{equation*}
\end{remark}

\begin{remark}\label{perfect-no-gv}
If our realization is perfect, then~$\mu(\one_z)$
and~$\mu(\one-\one_z)$ are disjoint projections.  From that it follows
easily that the only special good vector is the null vector.
\end{remark}

\begin{remark} If~\eqref{eq:sgv} holds with $z=e$, then the second
  inequality gives $v=0$. Therefore the definition of special good
  vector would be equivalent if we considered only nontrivial~$z$.
\end{remark}

\begin{lemma}\label{gv-linear}
  The good vectors make up a linear subspace of~$\mcH$.
\end{lemma}
\begin{proof}
 $(u+v)\otimes\overline{(u+v)}\leq
2(u\otimes\bar u+v\otimes\bar v)$.
\end{proof}

Recall that $\Gamma(x)\subseteq\Gamma$ is the subset of reduced words
which start with the reduced word for~$x$. Likewise, let
$\tilde\Gamma(x)$ be the subset of reduced words which end with
the reduced word for~$x$.
\begin{lemma}\label{potenzaT}Let $L=(L)_a$ be an $|A|$-tuple of operators
in $\mcB(\mcH)$ and let~$n\geq1$. Then
\begin{equation}\label{potenzaT0}
(\mcT^nL)_a=\sum_{b}
\sum_{\substack{x\in\Gamma\st|x|=n\\
        \text{$x\in\Gamma(a)$, $x\notin\tilde\Gamma(b^{-1})$}}}
\pi(x)L_b\pi(\inv x)\period
\end{equation}
\end{lemma}
\begin{proof}
For $B\in\mcB(\mcH)$ let
$P(x)B=\pi(x)B\pi(x)^{-1}$. Using this notation one has
\begin{equation*}
(\mcT L)_a=
\sum_{b}\mcT_{a,b}L_b=\sum_{b\neq a^{-1}}P(a)L_b\period
\end{equation*}
Now use induction. For $n=1$ one has
\begin{equation*}
  (\mcT^1)_{a,b} 
  =(\mcT)_{a,b} 
  =\left\{\begin{aligned}
      P(a) \quad& \text{if $b\neq a^{-1}$}\\
      0 \quad&\text{if $b=a^{-1}$}\\
  \end{aligned}\right\}
  =\sum_{\substack{|x|=1\\
      \text{$x\in\Gamma(a)$, $x\notin\tilde\Gamma(b^{-1})$}}}
  P(x) 
  \period
\end{equation*}
For $n>1$
\begin{multline*}
(\mcT^n)_{a,b}=
  \sum_c \mcT_{a,c}\mcT^{n-1}_{c,b}=\sum_{c\neq a^{-1}}P(a)\mcT^{n-1}_{c,b}\\
  =\sum_{c\neq a^{-1}}\sum_{\substack{|x|=n-1\\
       \text{$x\in\Gamma(c)$, $x\notin\tilde\Gamma(b^{-1})$}}}P(a)P(x)=
  \sum_{\substack{|y|=n\\ 
      \text{$y\in\Gamma(a)$, $y\notin\tilde\Gamma(b^{-1})$}}}
    P(y)
  \period
\end{multline*}
Thus
\begin{align*}
(\mcT^n L)_a &=\sum_b\mcT^n_{a,b}L_b \\
&=\sum_{b}\sum_{\substack{|y|=n \\
      \text{$y\in\Gamma(a)$, $y\notin\tilde\Gamma(b^{-1})$}}}
    P(y)L_b 
&=
 \sum_b\sum_{\substack{|y|=n \\
      \text{$y\in\Gamma(a)$, $y\notin\tilde\Gamma(b^{-1})$}}}
 \pi(y)L_b\pi(y)^{-1}
 \period
\end{align*}

\end{proof}

\begin{lemma}\label{l:good-stable}
The set of good vectors is stable under $\pi(\Gamma)$.
\end{lemma}
\begin{proof} Let $E$~be the vector with $E_a=v\otimes\bar
v$ and $E'$ the vector with $E'_a=\pi(y)v\otimes\overline{\pi(y)v}$.
If~$v$ is good, then we have $\mcT^NE\leq CF$ for some~$N$, hence
$(\mcT^nE)_a\leq CF_a$ for all~$a\in A$, and all~$n\geq N$.  According
to~\eqref{potenzaT0} this implies that
$\pi(x)v\otimes\overline{\pi(x)v}\leq CF_a$ whenever $x\in\GG(a)$ and
$|x|\geq N$. Consequently
$\pi(x)\pi(y)v\otimes\overline{\pi(x)\pi(y)v}\leq CF_a$ whenever
$x\in\GG(a)$ and~$|x|\geq N+|y|$.  Use~\eqref{potenzaT0} again to
deduce that $(\mcT^{N+|y|} E')_a\leq C'F_a$.
\end{proof}

\begin{lemma}\label{gvb} If $v$~is a good vector, then there
  exists~$C=C(v)>0$, independent of~$n$ so that
  \begin{equation*}
    \sum_{|x|=n} \pi(x)v\otimes\overline{\pi(x)v} \leq C\ID
    \period
  \end{equation*}
\end{lemma}
\begin{proof}
  By Remark~\ref{potenza-good-vector} there exist $C>0$, $N>0$ so that
  $CF\geq\mcT^n E$ for every~$n\geq N$ where $E_a=v\otimes\bar v$ for
  every $a$. The finite number of values of~$n$ with~$n<N$ create no
  difficulty, so assume $n\geq N$ and $n\geq 1$. Sum over~$a$ the
  inequalities $CF_a\geq (\mcT^n E)_a$ and apply Lemma~\ref{potenzaT}
  to get
  \begin{multline*}
    C\ID=C\sum_{a\in A} F_a \geq
     \sum_{a,b\in A}\sum_{\substack{|x|=n\\ 
        \text{$x\in\Gamma(a)$, $x\notin\tilde\Gamma(b^{-1})$}}}
    \pi(x)(v\otimes\bar v)\pi(x)^{-1} \\
   =    \sum_{a,b\in A}\sum_{\substack{x\in\Gamma\st |x|=n\\ 
        \text{$x\in\Gamma(a)$, $x\notin\tilde\Gamma(b^{-1})$}}}
   \pi(x)v\otimes\overline{\pi(x)v}
   = q\sum_{|x|=n}  \pi(x)v\otimes\overline{\pi(x)v}
  \end{multline*}
  where $q+1=|A|$.
\end{proof}

\begin{corollary}\label{gvb-4}
  If $v$~is a good vector, then there
  exists~$C=C(v)>0$, independent of~$n$ and~$w$ so that
  \begin{equation*}
    \sum_{|x|=n} 
     |\langle w,\pi(x) v\rangle|^2
     \leq C\norm{w}^2
     \period
  \end{equation*}
\end{corollary}

What about the existence of good vectors?
\begin{lemma}\label{sgv-is-gv}
  If $v$~is a special good vector then
\begin{itemize}
\item There exist $C>0$, $N>0$ so that if $a\in A$,
  $x\in\Gamma$,$|ax|=1+|x|$, and $|x|\geq N$, then
  $\pi(ax)(v\otimes\bar v)\pi(ax)^{-1}\leq C\mu(\one_a)=CF_a$.
\item $v$ is a good vector.
\end{itemize}
\end{lemma}
\begin{proof}
Choose~$C$ and~$z$ so that~\eqref{eq:sgv} holds.  Choose $N=|z|$.  For
the first assertion, note that if $x\notin\tilde\Gamma(z^{-1})$, then
\begin{multline*}
\pi(ax)(v\otimes\bar v)\pi(ax)^{-1}
\leq C\pi(ax)\mu(\one_z)\pi(ax)^{-1} \\
=C\mu(\lambda(ax)\one_z)
\leq C\mu(\one_a)=CF_a
\end{multline*}
and contrariwise if $x\in\tilde\Gamma(z^{-1})$, then
\begin{multline*}
\pi(ax)(v\otimes\bar v)\pi(ax)^{-1}
\leq C\pi(ax)\mu(\one-\one_z)\pi(ax)^{-1} \\
=C\mu(\lambda(ax)(\one-\one_z))
\leq C\mu(\one_a)=CF_a
  \period
\end{multline*}
The second assertion now follows from Lemma~\ref{potenzaT} since,
putting $E_a=v\otimes\bar v$ for each~$a\in A$,
\begin{equation*}
(\mcT^{N+1} E)_a=
  \sum_{\substack{
      \text{$x\in\Gamma$, $b\in A$} \\
      \text{$|ax|=1+|x|=N+1$, $x\notin\tilde\Gamma(b^{-1})$}}}
\pi(ax)(v\otimes\bar v)\pi(ax)^{-1}\leq CF_a
\end{equation*}
for a new value of~$C$.
\end{proof}

\begin{lemma}\label{l:sgvinv}
The set of special good vectors is stable under $\pi(\Gamma)$.
\end{lemma}
\begin{proof} If~$v$ satisfies~\eqref{eq:sgv} then $\pi(x)v$ satisfies
\begin{equation*}
\begin{aligned}
(\pi(x)v)\otimes\overline{(\pi(x)v)}
&=\pi(x)(v\otimes\bar v)\pi(x^{-1}) \\
&\leq C\pi(x)\mu(\one_z)\pi(x^{-1})
= C\mu(\lambda(x)\one_z)\;,
 \\
(\pi(x)v)\otimes\overline{(\pi(x)v)}
&=\pi(x)(v\otimes\bar v)\pi(x^{-1}) \\
&\leq C\pi(x)\mu(\one-\one_z)\pi(x^{-1})
= C\mu(\lambda(x)(\one-\one_z))\period
\end{aligned}
\end{equation*}
Now observe that the pair $\{\one_z,\one-\one_z\}$ is translated
by~$\lambda(x)$ to another such pair.  If
$x\notin\tilde\Gamma(z^{-1})$ then the translated pair is
$\{\one_{xz},\one-\one_{xz}\}$; if $x\in\tilde\Gamma(z^{-1})$ and if
$z=wa$ with $|z|=|w|+1$, then the translated pair is
$\{\one-\one_{xw},\one_{xw}\}$. This is easiest to understand by drawing
diagrams of the tree which is the Cayley graph of~$\Gamma$.
\end{proof}

\begin{lemma}\label{l:est}
Let $Q\in\mcB(\mcH)$ be a nonnegative operator and let
$u\in\mcH$.   Then
\begin{equation*}
(Q^{1/2}u)\otimes\overline{(Q^{1/2}u)}
  \leq \|{u}\|^2Q
  \period
\end{equation*}
\end{lemma}

\begin{proof}
Let $w\in\mcH$.  Then
\begin{multline*}
\langle{(Q^{1/2}u)\otimes\overline{(Q^{1/2}u)}w},{w}\rangle
  = |\langle{Q^{1/2}u},{w}\rangle|^2 \\
  = |\langle{u},{Q^{1/2}w}\rangle|^2
  \leq \|{u}\|^2 \|{Q^{1/2}w}\|^2
  = \|{u}\|^2 \langle{Qw},{w}\rangle
  \period
\end{multline*}
\end{proof}

\begin{lemma}\label{l:gsgv}
For any $u\in\mcH$ and $z\in\Gamma$,
\begin{equation*}
v=(\mu(\one_z)-\mu(\one_z)^2)^{1/2}u
\end{equation*}
is a special good vector.
\end{lemma}

\begin{proof} Note that
$\mu(\one_z)+\mu(\one-\one_z)=\mu(\one)=\ID$, hence
\begin{equation*}
  0\leq\mu(\one_z)\leq\ID
  \comma
    \qquad\qquad
  0\leq\mu(\one-\one_z)\leq\ID
  \period  
\end{equation*}
By Lemma~\ref{l:est} we have $v\otimes\bar
v\leq\|{u}\|^2(\mu(\one_z)-\mu(\one_z)^2)$.  
Now use
\begin{align*}
  \mu(\one_z)-\mu(\one_z)^2 &\leq\mu(\one_z)
    \comma \\
\mu(\one_z)-\mu(\one_z)^2
  &=(\ID-\mu(\one_z))-(\ID-\mu(\one_z))^2 \\
  &=\mu(\one-\one_z)-\mu(\one-\one_z)^2
\leq\mu(\one-\one_z)
\period
\end{align*}
\end{proof}

\begin{definition} Let $\mcH_B\subseteq\mcH$ consist of those vectors
which are orthogonal to all special good vectors and let
$\mcH_G=\mcH\ominus\mcH_B$. Thus $\mcH_G$~is the closure of the linear
span of the special good vectors.
\end{definition}

\begin{proposition}\label{gv-dense}
  $\mcH_G$ contains a dense linear subspace made up of good vectors.
\end{proposition}
\begin{proof}
This follows from Lemmas~\ref{sgv-is-gv} and~\ref{gv-linear}.
\end{proof}

\begin{proposition}\label{l:hb}\ 

\begin{enumerate}
\item $\mcH_B$ is a closed linear subspace.
\item $\mcH_B$ is invariant under~$\pi(\Gamma)$. 
\item $w\in\mcH$ belongs to $\mcH_B$ if and only if
$\mu(\one_z)w=\mu(\one_z)^2w$ for all $z\in\Gamma$.
\item For $w\in\mcH_B$,
$\mu(\one_y)\mu(\one_z)w=0$ whenever $\bGG(y)$ and~$\bGG(z)$
are disjoint.
\item For $w\in\mcH_B$, $\mu(G_1)\mu(G_2)w=\mu(G_1G_2)w$ for
$G_1$, $G_2\in C(\bGG)$.
\item $\mcH_B$ is stable under the action of~$\mu(C(\bGG))$.
\end{enumerate}
\end{proposition}
\begin{proof}
The first assertion is trivial.  The second assertion follows from
Lemma~\ref{l:sgvinv}.  For the third assertion, first suppose that
$w\in\mcH_B$.  Then by Lemma~\ref{l:gsgv} $\langle
w,(\mu(\one_z)-\mu(\one_z)^2)^{1/2}u\rangle=0$ for any
$u\in\mcH$, hence $(\mu(\one_z)-\mu(\one_z)^2)^{1/2}w=0$, hence
$(\mu(\one_z)-\mu(\one_z)^2)w=0$.

Conversely, suppose that $\mu(\one_z)w=\mu(\one_z)^2w$ for
all~$z\in\Gamma$ and suppose that $v$~is a special good vector
satisfying~\eqref{eq:sgv} for some particular $z\in\Gamma$.
Then $\langle v,w\rangle=\langle v,\mu(\one_z)w\rangle +\langle
v,\mu(\one-\one_z)w\rangle$ and
\begin{align*}
|\langle v,\mu(\one_z)w\rangle|^2
&=\langle (v\otimes\bar v)\mu(\one_z)w,\mu(\one_z)w\rangle \\
&\leq C\langle\mu(\one-\one_z)\mu(\one_z)w,\mu(\one_z)w\rangle
  \comma \\
|\langle v,\mu(\one-\one_z)w\rangle|^2
&=\langle (v\otimes\bar v)\mu(\one-\one_z)w,\mu(\one-\one_z)w\rangle \\
&\leq C\langle\mu(\one_z)\mu(\one-\one_z)w,\mu(\one-\one_z)w\rangle
\period
\end{align*}
Now use
$\mu(\one-\one_z)\mu(\one_z)w=(\mu(\one_z)-\mu(\one_z)^2)w=0$
in both terms.

In the fourth assertion, we assume that $\bGG(y)$ and~$\bGG(z)$
are disjoint, hence that $\one_y\leq\one-\one_z$.
\begin{align*}
\langle \mu(\one_y)^{1/2}\mu(\one_z)w,
  \mu(\one_y)^{1/2}\mu(\one_z)w \rangle
&=\langle \mu(\one_y)\mu(\one_z)w, \mu(\one_z)w \rangle \\
&\leq \langle \mu(\one-\one_z)\mu(\one_z)w,
  \mu(\one_z)w \rangle
  =0
  \period
\end{align*}
Hence $\mu(\one_y)^{1/2}\mu(\one_z)w=0$, hence $\mu(\one_y)\mu(\one_z)w=0$.

Let $w\in\mcH_B$.  Suppose that for some $n>0$ each of $G_1$, $G_2\in
C(\bGG)$ is of the form $\sum_{|z|=n} c_z\one_z$.  Then
$\mu(G_1)\mu(G_2)w=\mu(G_1G_2)w$ follows from the third and
fourth assertions and linearity.  Taking limits, we see that this
formula is valid for arbitrary $G_1$, $G_2\in C(\bGG)$.

Finally, let $w\in\mcH_B$ and $G\in C(\bGG)$.  For any $z\in\Gamma$
\begin{equation*}
\mu(\one_z)^2\mu(G)w
=\mu(\one_z)\mu(\one_zG)w
=\mu(\one_zG)w=\mu(\one_z)\mu(G)w
\end{equation*}
and according to the criterion from the third assertion, this shows
that $\mu(G)w\in\mcH_B$.
\end{proof}

\begin{definition} 
Let~$\mcH'_B$ and~$\mcH'_G$~be the closures of
$\pi'(C(\bGG))\iota(\mcH_B)$ and $\pi'(C(\bGG))\iota(\mcH_G)$)
in~$\mcH'$.  By Proposition~\ref{l:hb} $\mcH_B$ and~$\mcH_G$ are stable
under~$\pi(\Gamma)$ and consequently each of~$\mcH'_B$ and~$\mcH'_G$ is
stable under $\pi'(\Gamma\ltimes C(\bGG))$.  Let $\pi_B$ denote the
restriction of~$\pi$ to $\mcH_B$ and~$\pi_G$ the restriction to
$\mcH_G$.  Likewise let $\pi'_B$~be the restriction of~$\pi'$ to
$\mcH'_B$ and~$\pi'_G$ the restriction to $\mcH'_G$. Let
$\iota_B:\mcH_B\to\mcH_B'$ and $\iota_G:\mcH_G\to\mcH_G'$ be the
respective restrictions of~$\iota$.  They are boundary realizations
of~$\mcH_B$ and~$\mcH_G$ respectively.
\end{definition}

\begin{corollary} $\mcH'=\mcH'_B\oplus\mcH'_G$.
\end{corollary}
\begin{proof}
\begin{multline*}
\langle \pi'(C(\bGG))\iota(\mcH_B),
   \pi'(C(\bGG))\iota(\mcH_G) \rangle \\
=\langle \iota^*\pi'(C(\bGG))\iota \mcH_B, \mcH_G \rangle
=\langle \mu(C(\bGG))\mcH_B, \mcH_G \rangle
=\langle \mcH_B,\mcH_G \rangle=0
\end{multline*}
using the last assertion of Proposition~\ref{l:hb}.  This shows that
$\mcH'_B$ and~$\mcH'_G$ are perpendicular.  Consequently, their direct
sum is a closed subspace of~$\mcH'$.  That sum contains
$\pi'(C(\bGG))\iota(\mcH)$ and is consequently total.
\end{proof}

\begin{corollary} \label{gvd} $(\iota_B,\mcH'_B)$  is a perfect
  realization of~$\mcH_B$.  Consequently $\mcH'_B=\iota_B(\mcH_B)$.
\end{corollary}
\begin{proof}
  Use the fifth assertion of Proposition~\ref{l:hb} and
  Lemma~\ref{perfect-by-mu}.
\end{proof}

We never use this last lemma, but it rounds out the
picture.
\begin{lemma} All good vectors lie in~$\mcH_G$.
\end{lemma}
\begin{proof}[Sketch of proof]
  Suppose that $v=v_B+v_G\in\mcH$ with~$v_B\in\mcH_B$
  and~$v_G\in\mcH_G$. Let $E$~be the vector with
  \begin{equation*}
    E_a=v\otimes\bar v=(v_B+v_G)\otimes\overline{(v_B+v_G)}
    =\begin{pmat}
     v_B\otimes\bar v_B & v_B\otimes\bar v_G \\ 
     v_G\otimes\bar v_B & v_G\otimes\bar v_G
    \end{pmat}
    \period
  \end{equation*}
  Suppose $v$~is good for~$\iota$. This translates to $\mcT^n E\leq
  CF$. Calculate both sides as block matrices.  Looking at the upper
  left hand block shows that $v_B$~is good
  for~$\iota_B$. Corollary~\ref{gvd} says that $\iota_B$ is perfect,
  and so Remark~\ref{perfect-no-gv} says that $v_B=0$.
\end{proof}

\section{Main proofs}\label{main-proofs}

\begin{lemma} \label{NotPosL}
Assume that $\mu:C(\bGG)\to\mcB(\mcH)$ is a $^*$-map
satisfying $\pi(x)\mu(G)\pi(x^{-1})=\mu(\lambda(x)G)$.
 Suppose also that $\norm{\mu(G)}\leq
C\norm{G}_\infty$. If $\mu$ is \emph{not} a positive map, then there
is some~$a\in A$ such that $\mu(\one-\one_a)$ is \emph{not}
positive~semidefinite.
\end{lemma}
\begin{proof}
Observe first that a norm-continuous $^*$-map
$\mu:C(\bGG)\to\mcB(\mcH)$ is positive if and only if $\mu(\one_z)$ is
positive for each $z\in\Gamma$.  Hence, if $\mu$ is not positive,
there exists $z\in\GG$ and $v\in\mcH$ such that
$\langle\mu(\one_z)v,v\rangle<0$.  Denote by $c$ the last letter of
$z$ and let $w=\pi(z^{-1})v$.  One has
\begin{multline*}
\langle\mu(\one_z)v,v\rangle=
\langle\pi(z^{-1})\mu(\one_z)\pi(z)w,w\rangle=\\
\langle\mu(\lambda(z^{-1})\one_z) w,w\rangle=
\langle\mu(\one-\one_{c^{-1}}) w,w\rangle<0
\end{multline*}
implying that $\mu(\one-\one_{c^{-1}})$ is not positive semidefinite.
\end{proof}

\subsection{Oddity} In this subsection, the reigning hypotheses are as
follows:
\begin{itemize}
\item $\pi:\Gamma\to\mcU(\mcH)$ is a unitary representation of~$\Gamma$.
\item $\pi':\cp\to\mcB(\mcH')$ is an irreducible representation of~$\cp$.
\item $\iota:\mcH\to\mcH'$ is an imperfect boundary realization
  of~$\pi$, i.e. $\iota$~is a~$\Gamma$-map which is an isometry but is
  \emph{not} a unitary isomorphism.
\item As per Section~\ref{imfS}, $\mu$ and~$F$ are associated to~$\iota$.
\item The~\eqref{FTC} holds for~$F$.
\end{itemize}

Let~$\mcH_1=\mcH$ and $\mcH_2=\mcH'\ominus\iota(\mcH)$;
let~$\iota_1=\iota$ and let $\iota_2:\mcH_2\to\mcH'$ be the inclusion
map.  This sets up the natural symmetry between~$\mcH=\mcH_1$
and~$\mcH_2$, with $\mcH'=\iota_1(\mcH_1)\oplus\iota_2(\mcH_2)$. Note
that $\mcH_2$ is not stable under~$\cp$, but it is stable
under~$\Gamma$; set~$\pi_1=\pi:\Gamma\to\mcU(\mcH_1)$ and
let~$\pi_2:\Gamma\to\mcU(\mcH_2)$ be the~$\Gamma$-action on~$\mcH_2$
obtained from the $\cp$~representation. The only difference between
the two direct summands is that we have assumed~\eqref{FTC}
for~$\iota_1$. This asymmetry is only apparent.

\begin{proposition}\label{odd-symm} Let $F_1=F$ and~$F_2$ be
  associated with~$\iota_1$ and~$\iota_2$ respectively.  Suppose that
  the block matrix for~$\pi'(\one_a)$ is given by
  \begin{equation*}
    \begin{pmatrix}
      \pi'_{11}  & \pi'_{12} \\
      \pi'_{21}  & \pi'_{22}
    \end{pmatrix}(\one_a) \period
  \end{equation*}
  Then
  \begin{equation*}
    (F_1,F_1) = \sum_a \norm{\pi'_{21}(\one_a)}_{HS}^2
              = \sum_a \norm{\pi'_{12}(\one_a)}_{HS}^2
              =(F_2,F_2) \period
  \end{equation*}
\end{proposition}
Besides establishing the~\eqref{FTC} for~$\iota_2$, this shows that
the~\eqref{FTC}, which we are assuming here, is equivalent to the
corresponding condition in the statement of the Oddity~Theorem in
Section~\ref{intro}.
\begin{proof}
  The middle equality is trivial because
  $\pi'_{12}(\one_a)={{\pi'}_{21}(\one_a)}^*$, which holds because
  $\pi'(\one_a)$ is self adjoint.  By definition
  \begin{equation*}
    (F_1,F_1)=\sum_{a,d\st a\neq d}\TR(\pi'_{11}(\one_a),\pi'_{11}(\one_d))=
    \sum_a\TR(\pi'_{11}(\one_a),\pi'_{11}(\one-\one_a)) \period
\end{equation*}    
  Since $\pi'_{11}(\one)=\ID$, $\pi'_{11}(\one_a)$ and
  $\pi'_{11}(\one-\one_a)$ commute, so
  \begin{equation*}
    \TR(\pi'_{11}(\one_a),\pi'_{11}(\one-\one_a))=
    \tr(\pi'_{11}(\one_a)\pi'_{11}(\one-\one_a)) \period
  \end{equation*}
One calculates
\begin{equation*}
  0=(\pi'(\one_a)\pi'(\one-\one_a))_{11}
  =\pi'_{11}(\one_a)\pi'_{11}(\one-\one_a)
    +\pi'_{12}(\one_a)\pi'_{21}(\one-\one_a) \period
\end{equation*}

Since $\pi'(\one)=\ID$ we have $\pi'_{21}(\one)=0$, whence
\begin{equation*}
  \pi'_{11}(\one_a)\pi'_{11}(\one-\one_a)=
  \pi'_{12}(\one_a)\pi'_{21}(\one_a) \period
\end{equation*}
Take traces of both sides and use
$\pi'_{12}(\one_a)={{\pi'}_{21}(\one_a)}^*$ to get $(F_1,F_1) = \sum_a
\norm{\pi'_{21}(\one_a)}_{HS}^2$. The formula for~$(F_2,F_2)$ follows in
exactly the same way.
\end{proof}

\begin{proposition}\label{OddLimP}
  Assume the reigning hypotheses of this subsection.  Let $F_0$~be an
  $|A|$-tuple of positive semidefinite operators satisfying $(\mcT^N
  F_0)_a\leq CF_a$ for some fixed integer~$N\geq 0$, some fixed $C>0$,
  and for all~$a\in A$. Then
  \begin{equation}\label{OddLimE}
    \wklim_{\epsilon\to 0+} \epsilon\sum_{n\geq 0} e^{-\epsilon n} \mcT^n
       F_0 = \frac{(F_0,F)}{(F,F)} F
    \period
  \end{equation}
\end{proposition}
\begin{proof}
  Note that $(F_0,F)=(F_0,\mcT^N F)=(\mcT^N F_0,F)\leq C(F,F)<\infty$.
  Use Proposition~\ref{limP} to see that the limit exists for some
  subsequence $\epsilon_j\to 0+$. Use Corollary~\ref{sub-rep-irr} to
  see that the limit must be of the form~$tF$. Again by
  Proposition~\ref{limP} the value of~$t$ must be as shown. Since
  any subsequence such that the limit exists gives the same limit,
  that limit must be valid for $\epsilon\to 0+$.
\end{proof}

\begin{remark}
If the hypotheses of Proposition~\ref{OddLimP} hold, except
that instead of the~\eqref{FTC} one has~$(F,F)=\infty$, and
if~$(F_0,F)<\infty$, then a similar argument shows that the limit is
zero.
\end{remark}

\begin{theorem} Assume the reigning hypotheses of this
  subsection. Let $\pi'_N:\cp\to\mcB(\mcH_N')$ be a boundary
  representation and~$\iota_N:\mcH\to\mcH_N'$ a
  boundary intertwiner.  Let~$\mu_N$ and $F_N$ be
  associated to~$\iota_N$, as in Section~\ref{imfS}. Then~$\mu_N$ is a
  scalar multiple of~$\mu$.
\end{theorem}
\begin{proof} Any mention of ``good vectors'' in this proof means
  good vectors in~$\mcH$ relative to the boundary realization~$\iota$;
  we never consider good vectors relative to~$\iota_N$ or to any other
  boundary intertwiner. Suppose the good vectors weren't dense
  in~$\mcH$. According to Corollary~\ref{gvd} this would mean there
  was a nonzero $\Gamma$-invariant subspace $\mcH_B\subseteq\mcH$
  such that $\iota|_{\mcH_B}$ was a perfect realization; hence
  $\iota(\mcH_B)$~would be a $\cp$-subspace of~$\mcH'$. Since
  $\mcH'$~is irreducible, this would imply that $\iota(\mcH_B)$~was
  all of~$\mcH'$, a contradiction since $\iota$~isn't surjective.

  Let $t=\max\{t\geq 0 \st \text{$\mu_N-t\mu$ is a positive map}\}$
  and let $\mu_t=\mu_N-t\mu$. As per Section~\ref{imfS}, we
  find~$\iota_t$ and~$F_t$ associated to~$\mu_t$. After several steps,
  we shall show that~$\mu_t=0$, and it will follow that~$\mu_N=t\mu$,
  proving the theorem.

  Fix any~$\delta>0$. From the definition of~$t$ it follows that
  $\mu_t-\delta\mu$ is \emph{not} a positive map. As per
  Lemma~\ref{NotPosL} choose~$a\in A$ so that
  $(\mu_t-\delta\mu)(\one-\one_a)$ is not positive~semidefinite,
  and a good vector~$u\in\mcH$ so that 
  \begin{equation}\label{wOddE}
    \langle  (\mu_t-\delta\mu)(\one-\one_a) u,u \rangle<0 \period
  \end{equation}
  Define~$F_0$ by
  \begin{align*}
    (F_0)_{a}&=u\otimes\bar u &
    (F_0)_c&=0 \qquad\text{for $c\neq a$.}
  \end{align*}
  Then it follows from~\eqref{wOddE} that
  \begin{multline}\label{F0FE}
    (F_0,F_t)=\TR(u\otimes\bar u,\mu_t(\one-\one_a)) \\
       =\langle \mu_t(\one-\one_a) u,u \rangle
       <\langle \delta\mu(\one-\one_a) u,u \rangle
       =\delta (F_0,F)
    \period
  \end{multline}

  According to Proposition~\ref{OddLimP}
  \begin{equation*}
    F_L=\wklim_{\epsilon\to 0+} \epsilon\sum_{n\geq 0} e^{-\epsilon n} \mcT^n
       F_0 = \frac{(F_0,F)}{(F,F)} F
  \end{equation*}
  and this is a nonzero multiple of~$F$ because of~\eqref{F0FE}.
  Multiplying~$u$ by a scalar, we may assume the limit~$F_L$ is~$F$
  itself.

  Now using Proposition~\ref{limP}, Corollary~\ref{FatouC}, the
  identities $(\mcT F_1,F_2)=(F_1,\mcT F_2)$, $\mcT F_t=F_t$, and
  $\mcT F=F$, equation~\eqref{F0FE} and again Proposition~\ref{limP},
  we find:
  \begin{align*}
    (F,F_t)&=(F_L,F_t)
       =(\wklim_{\epsilon\to 0+} \epsilon\sum_{n\geq 0}
         e^{-\epsilon n} \mcT^n F_0,F_t) \\
       &\leq \liminf_{\epsilon\to 0+} (\epsilon\sum_{n\geq 0}
         e^{-\epsilon n} \mcT^n F_0,F_t) \\ 
       &=\liminf_{\epsilon\to 0+}\epsilon\, (F_0,F_t)/(1-e^{-\epsilon})
       = (F_0,F_t) < \delta (F_0,F) \\
       &=\delta \lim_{\epsilon\to 0+} \epsilon (F_0,F)/(1-e^{-\epsilon})
       =\delta(\wklim_{\epsilon\to 0+} \epsilon\sum_{n\geq 0}
         e^{-\epsilon n} \mcT^n F_0,F) \\ 
       &=\delta(F_L,F)=\delta(F,F)
     \period
  \end{align*}
  Send $\delta\to 0+$ to conclude that $(F,F_t)=0$.

  Unless $\mu_t=0$, you may apply Proposition~\ref{NotPosL} to~$-\mu_t$
  to find~$b\in A$ such that $-\mu_t(\one-\one_b)$~is not positive
  definite. Then choose a good vector~$v\in\mcH$ so that $\langle
  \mu_t(\one-\one_b)v,v)\rangle>0$, and define~$F_1$ by
  \begin{align*}
    (F_1)_{b}&=v\otimes\bar v &
    (F_1)_c&=0 \qquad\text{for $c\neq b$.}
  \end{align*}
  One may then calculate $(F_1,F_t)=\langle \mu_t(\one-\one_b)v,v
  \rangle>0$.  Since $v$~is good, one knows that for $n$~large enough
  $\mcT^n F_1\leq CF$. This leads to the contradiction
  \begin{equation*}
    0<(F_1,F_t)=(F_1,\mcT^n F_t)=(\mcT^n F_1,F_t)\leq C(F,F_t)=0
    \period
  \end{equation*}
  \end{proof}

\begin{corollary}\label{irreducible} Under the above hypotheses, $\pi=\pi_1$ and~$\pi_2$
  are irreducible $\Gamma$-representations. 
\end{corollary}
\begin{proof} By Proposition~\ref{odd-symm} $\pi_2$ 
satisfies the same hypothesis as $\pi_1$, hence it is enough to prove
the assertion for $\pi_1$.  Assume, by contradiction, that $\pi_1$ is
reducible. Split $\mcH=\mcH_0\oplus\mcH_0^\perp$ into the direct sum
of $\pi_1$ invariant subspaces and let $P_0$ be the projection onto
$\mcH_0$. Let $\iota'=\iota P_0$: then $\iota'$ is another boundary
intertwiner for $\pi_1$, and it is clearly not equivalent to any scalar
multiple of~$\iota$.
\end{proof}

\begin{corollary} Under the above hypotheses, $\pi=\pi_1$ and~$\pi_2$
  are inequivalent as $\Gamma$-representations. 
\begin{proof}
Assume, by contradiction, that $U\pi_1=\pi_2U$  for some unitary
$U:\mcH_1\to\mcH_2$. Let  $\iota'= \iota_2U$. Then $\iota'$ is another
boundary realization of $\pi_1$ which must be equivalent to $\iota$,
that is
 $\iota'= J\iota$ where $J$ intertwines
$\pi'$ to itself. Since $\pi'$ is irreducible $J$ must be a scalar,
which
is impossible.
\end{proof}
\end{corollary}

\begin{corollary} Under the above hypotheses, for $j=1$ or~$2$, any
  boundary realization of~$\pi_j$ is equivalent to~$\iota_j$.
\end{corollary}
\begin{proof}
 By Proposition~\ref{odd-symm} $\pi_2$ 
satisfies the same hypothesis as $\pi_1$.
\end{proof}

\subsection{Duplicity}
The proof of the Oddity~Theorem, in the previous subsection, and the
proof of the Duplicity~Theorem, in this subsection, are closely
parallel. In this subsection the reigning hypotheses are as follows:
\begin{itemize}
\item $\pi:\Gamma\to\mcU(\mcH)$ is a unitary representation of~$\Gamma$.
\item We have two irreducible representations
  $\pi'_\pm:\cp\to\mcB(\mcH_\pm')$, which are inequivalent as
  representations of~$\cp$.
\item We have two perfect boundary realizations of~$\pi$,
  $\iota_\pm:\mcH\to\mcH_\pm'$.
\item Let $\iota:\mcH\to\mcH_+\oplus\mcH_-$ be defined by
  $\iota=\frac{1}{\sqrt 2}(\iota_+\oplus\iota_-)$.
  As per Section~\ref{imfS}, $\mu$ and~$F$ are associated to~$\iota$.
\item The~\eqref{FTC} holds for~$F$.
\end{itemize}

By hypothesis, $\iota_\pm$ are perfect realizations, but clearly
$\iota$ is not. Does $\iota(\mcH)$ lie in some proper $\cp$-subspace
of~$\mcH_1\oplus\mcH_2$? According to Lemma~\ref{sum-two-irrin} it
does not.
Let~$\mu_\pm$ and $F_\pm$ be associated with~$\iota_\pm$. As explained
in subsection~\ref{dir-sums}, $\mu=\frac{1}{2}(\mu_++\mu_-)$ and
$F=\frac{1}{2}(F_++F_-)$. The~\eqref{FTC} for~$F$ is $(F,F)<\infty$.
Proposition~\ref{(F,F)=0} says that $(F_+,F_+)=(F_-,F_-)=0$. Thus,
the~\eqref{FTC} is equivalent to $(F_+,F_-)<\infty$. Since $\iota$~is
not perfect, Proposition~\ref{(F,F)=0} says that~$(F,F)>0$,
i.e. $(F_+,F_-)>0$.

\begin{proposition} The~\eqref{FTC} holds for~$F$ if and only if
  \begin{equation*}
    \norm{(F_+)_a(F_-)_b}_{HS}<\infty
  \end{equation*}
whenever $a,b\in A$, $a\neq b$.
\end{proposition}
\begin{proof}
By definition of the inner product $(F_+,F_-)$, the~\eqref{FTC} means
that $\TR((F_+)_a,(F_-)_b)<\infty$ whenever~$a\neq b$.  Because
$\iota_\pm$ are perfect realizations, the components of~$F_\pm$ are
projections. For a pair of projections one has
\begin{multline*}
  \TR((F_+)_a,(F_-)_b) =\tr((F_+)_a(F_-)_b(F_+)_a) \\
  =\tr((F_+)_a(F_-)_b((F_+)_a(F_-)_b)^*) =\norm{(F_+)_a(F_-)_b}_{HS}^2.
\end{multline*}
\end{proof}
This shows that the~\eqref{FTC} which we are assuming here is
equivalent to the finiteness condition in the statement of the
Duplicity~Theorem in Section~\ref{intro}.

\begin{proposition}\label{DupLimP}
  Assume the reigning hypotheses of this subsection.  Let $F_0$~be an
  $|A|$-tuple of positive operators satisfying $(\mcT^N F_0)_a\leq
  CF_a$ for some fixed integer~$N\geq 0$, some fixed $C>0$, and for
  all~$a\in A$. Then
  \begin{equation}\label{DupLimE}
    F_L=\wklim_{\epsilon\to 0+} \epsilon\sum_{n\geq 0} e^{-\epsilon n} \mcT^n
       F_0 = \frac{(F_0,F_-)F_++(F_0,F_+)F_-}{(F_+,F_-)}
    \period
  \end{equation}
\end{proposition} 
\begin{proof}
  Proposition~\ref{limP}, Corollary~\ref{SubMuP}, and
  Corollary~\ref{BdConvC} together with the finiteness of
  $(F_+,F_-)$.
\end{proof}

\begin{remark} If the hypotheses of Proposition~\ref{DupLimP} hold, except
that instead of the~\eqref{FTC} one has~$(F_+,F_-)=\infty$, and
if~$(F_0,F)<\infty$, then a similar argument shows that the limit is
zero.
\end{remark}

\begin{theorem}\label{duplicity} Assume the reigning hypotheses of this
  subsection. Let $\pi'_N:\cp\to\mcB(\mcH_N')$ be any boundary
  representation and suppose that~$\iota_N:\mcH\to\mcH_N'$ is a
  boundary intertwiner (a~$\Gamma$-map).  Let~$\mu_N$ and $F_N$ be
  associated to~$\iota_N$, as per
  Section~\ref{imfS}. Then~$\mu_N=t_+\mu_++t_-\mu_-$ for nonnegative
  coefficients~$t_\pm$.
\end{theorem}
\begin{proof} Any mention of ``good vectors'' in this proof means
  good vectors in~$\mcH$ relative to the boundary realization~$\iota$;
  we never consider good vectors relative to~$\iota_N$,
  to~$\iota_\pm$, or to any other boundary intertwiner.

  Suppose the good vectors weren't dense in~$\mcH$. According to
  Corollary~\ref{gvd} this would mean there was a nonzero
  $\Gamma$-invariant subspace $\mcH_B\subseteq\mcH$ such that
  $\iota|_{\mcH_B}$ was a perfect realization; hence
  $\iota(\mcH_B)$~would be a $\cp$-subspace of~$\mcH'$. Since
  $\mcH'_\pm$~are irreducible and inequivalent, the only possibilities
  for that subspace would be~$0$, $\mcH'_+$, $\mcH'_-$, and
  $\mcH'_+\oplus\mcH'_-$. Since $\iota(v)=\frac{1}{\sqrt
    2}(\iota_+(v),\iota_-(v))$, no nonzero vector of~$\iota(\mcH)$
  lies in any of the three proper $\cp$-subspaces.  Moreover,
  $\iota(\mcH)$ isn't all of $\mcH'_+\oplus\mcH'_-$, so its subspace
  $\iota(\mcH_B)$ can't be either. One concludes that the good vectors
  are dense in~$\mcH$.

  Let $t_+=\max\{t\geq 0 \st \text{$\mu_N-t\mu_+$ is a positive
    map}\}$ and then let $t_-=\max\{t\geq 0 \st
  \text{$\mu_N-t_+\mu_+-t\mu_-$ is a positive map}\}$.  Let
  $\mu_R=\mu_N-t_+\mu_+-t_-\mu_-$.  Note that for any~$\delta>0$,
  neither $\mu_R-\delta\mu_+$ nor $\mu_R-\delta\mu_-$ is a positive
  map. As per Section~\ref{imfS}, we find~$\iota_R$ and~$F_R$
  associated to~$\mu_R$. After several steps, we shall show
  that~$\mu_R=0$, and it will follow that~$\mu_N=t_+\mu_++t_-\mu_-$,
  proving the theorem.

  Fix any~$\delta>0$. Using Lemma~\ref{NotPosL} choose~$a\in A$ so
  that $(\mu_R-\delta\mu_+)(\one-\one_a)$ is not positive~semidefinite,
  and a good vector~$u\in\mcH$ so that
  \begin{equation}\label{wDupE}
    \langle  (\mu_R-\delta\mu_+)(\one-\one_a) u,u \rangle<0 \period
  \end{equation}
  Define~$F_0$ by
  \begin{align*}
    (F_0)_{a}&=u\otimes\bar u &
    (F_0)_c&=0 \qquad\text{for $c\neq a$.}
  \end{align*}
  Then it follows from~\eqref{wDupE} that
  \begin{multline}\label{F0FR}
    (F_0,F_R)=\TR(u\otimes\bar u,\mu_R(\one-\one_a)) \\
       =\langle \mu_R(\one-\one_a) u,u \rangle
       <\langle \delta\mu_+(\one-\one_a) u,u \rangle
       =\delta (F_0,F_+)
    \period
  \end{multline}
  In particular this shows that $(F_0,F_+)>0$.
  

  Now using Corollary~\ref{FatouC}, the identities $(\mcT
  F_1,F_2)=(F_1,\mcT F_2)$, $\mcT F_R=F_R$, and $\mcT F=F$,
  equation~\eqref{F0FR} and Corollary~\ref{BdConvC}, we find:
  \begin{align*}
    &\frac{(F_0,F_+)(F_-,F_R)}{(F_+,F_-)}
      \leq
        \left(\frac{(F_0,F_-)F_++(F_0,F_+)F_-}{(F_+,F_-)},F_R\right) 
      =(F_L,F_R) \\
       &\qquad
      =(\lim_{\epsilon\to 0+} \epsilon\sum_{n\geq 0}
      e^{-\epsilon n} \mcT^n F_0,F_R)
      \leq \liminf_{\epsilon\to 0+} (\epsilon\sum_{n\geq 0}
      e^{-\epsilon n} \mcT^n F_0,F_R) \\
      &\qquad
      =\liminf_{\epsilon\to 0+}\epsilon\, (F_0,F_R)/(1-e^{-\epsilon})
      = (F_0,F_R)
      < \delta (F_0,F_+) \\
      &\qquad
      =\delta \lim_{\epsilon\to 0+} \epsilon
      (F_0,F_+)/(1-e^{-\epsilon})
       =\delta(\lim_{\epsilon\to 0+} \epsilon\sum_{n\geq 0}
       e^{-\epsilon n} \mcT^n F_0,F_+)  \\
       &\qquad
       =\delta(F_L,F_+)
       =\left(\frac{(F_0,F_-)F_++(F_0,F_+)F_-}{(F_+,F_-)},F_+\right) 
       =\delta (F_0,F_+)
     \period
  \end{align*}
  Cancel the factor of~$(F_0,F_+)$ and send $\delta\to 0+$ to conclude
  that $(F_-,F_R)=0$. An identical argument shows that $(F_+,F_R)=0$.
  Consequently $(F,F_R)=0$.

  The following final step is identical to the corresponding step in
  the previous subsection.  Unless $\mu_R=0$, you may apply
  Proposition~\ref{NotPosL} to~$-\mu_R$ to find~$b\in A$ such that
  $-\mu_R(\one-\one_b)$~is not positive semidefinite. Then choose a
  good vector~$v\in\mcH$ so that $\langle
  \mu_R(\one-\one_b)v,v\rangle>0$, and define~$F_1$ by
  \begin{align*}
    (F_1)_{b}&=v\otimes\bar v &
    (F_1)_c&=0 \qquad\text{for $c\neq b$.}
  \end{align*}
  One may then calculate $(F_1,F_R)=\langle \mu_R(\one-\one_b)v,v
  \rangle>0$.  Since $v$~is good, it follows that for $n$~large enough
  $\mcT^n F_1\leq CF$. This leads to the contradiction
  \begin{equation*}
    0<(F_1,F_R)=(F_1,\mcT^n F_R)=(\mcT^n F_1,F_R)\leq C(F,F_R)=0
    \period
  \end{equation*}
  \end{proof}

\begin{corollary} Under the above hypotheses, $\pi$ is an irreducible
  $\Gamma$-representation. 
\end{corollary}
\begin{proof}
Argue as in the proof of Corollary~\ref{irreducible}.
\end{proof}
\begin{corollary} Under the above hypotheses, any nonzero
  boundary intertwiner of~$\pi$ is equivalent to~$s_+\iota_+$ for some
  $s_+>0$, to $s_-\iota_-$ for some $s_->0$, or to $s_+\iota_+\oplus
  s_-\iota_-$ for some $s_\pm>0$.
\end{corollary}
In the last case, the boundary intertwiner will be a boundary
realization (an isometry) if $s_+^2+s_-^2=1$.

\section{``Schur orthogonality''}\label{schur-sec}

The aim of this Section is to show how Proposition~\ref{DupLimP} can
be specialized using \emph{good vectors} to get formulae for limits of
sums of products of matrix coefficents.  The version of Schur
orthogonality discussed here involves matrix coefficients of one fixed
representation~$\pi$. Similar Schur orthogonality for coefficients of
two different representations is worth looking into, but we do not
know how to proceed.
\subsection{Duplicity}
We consider the case of duplicity first.  The reigning hypotheses are:
\begin{itemize}
\item $\pi:\Gamma\to\mcU(\mcH)$
  is a unitary representation of~$\Gamma$.
\item We have two irreducible representations
  $\pi'_\pm:\cp\to\mcB(\mcH_\pm')$, which are inequivalent as
  representations of~$\cp$.
\item We have two perfect boundary realizations of~$\pi$,
  $\iota_\pm:\mcH\to\mcH_\pm'$.
\item Let $\iota:\mcH\to\mcH_+\oplus\mcH_-$ be defined by
  $\iota=\frac{1}{\sqrt 2}(\iota_+\oplus\iota_-)$.
  As per Section~\ref{imfS}, associate~$\mu$ and~$F$ to~$\iota$.
\item The~\eqref{FTC} holds for~$F$.
\end{itemize}
Our aim is to prove Theorem~\ref{Schur}. In the statement of that
theorem appears a dense linear subspace
$\mcH^\infty\subseteq\mcH$. Let $\mcH^\infty$ be the subspace of good
vectors with respect to $\iota$. $\mcH^\infty$ is dense (see the proof
of Theorem~\ref{duplicity}), closed under linear combinations
(Lemma~\ref{gv-linear}) and $\GG$-invariant
(Lemma~\ref{l:good-stable}). Also in Theorem~\ref{Schur} appears a positive
constant $A_\pi$.  Let $A_\pi=1/(F_+,F_-)$.

It is clear that the restriction map from $C(\GG\sqcup\bGG)$ to
$C(\bGG)$ is a map of $C^*$-algebras. Abusing notation we
define, for any $G$ in $C(\GG\sqcup\bGG)$,
\begin{equation*}
\pi'_\pm(G)= \pi'_\pm(G|_{\bGG})\qquad
\text{and}\qquad\mu_\pm(G)=\mu_\pm
(G|_{\bGG})\period
\end{equation*}
For $G\in C(\GG\sqcup\bGG)$ and~$x\in\GG$ set
$G^*(x)=\overline{G(x^{-1})}$.  We must prove
  \begin{multline}\label{schur2}
  \lim_{\epsilon\to 0+} \epsilon
    \sum_{x\in\Gamma} e^{-\epsilon|x|} G(x)\tilde G^*(x)
      \langle v_1,\pi(x)v_3 \rangle
      \overline{\langle v_2,\pi(x)v_4 \rangle} \\
      =\frac{1}{(F_+,F_-)}\left(
        \langle \pi'_+(G)\iota_+ v_1,\iota_+v_2\rangle
        \overline{\langle \pi'_-(\tilde G)
          \iota_-v_3,\iota_-v_4\rangle}\right. \\
      +\left.
        \langle \pi'_-(G)\iota_-v_1,\iota_-v_2\rangle
        \overline{\langle \pi'_+(\tilde G)\iota_+v_3,
           \iota_+v_4\rangle} \right)
      \period
  \end{multline}
for any $G,\tilde{G}\in
C(\GG\sqcup\bGG)$, $v_1,v_2\in\mcH$, and $v_3,v_4\in\mcH^\infty$.
We need only prove~\eqref{schur2} with $v_1=v_2=w\in\mcH$ and
$v_3=v_4=v\in\mcH^\infty$. Once this is done, polarization, first with
respect to~$v_1$ and~$v_2$, then with respect to~$v_3$ and~$v_4$ will
give~\eqref{schur2} in generality.

Taking into account the definition of $\mu_\pm$ we see that with
$v_1=v_2=w$ and $v_3=v_4=v$ formula~\eqref{schur2}
becomes:
\begin{multline}\label{schurbdry}
\lim_{\epsilon\to 0+}
\epsilon\sum_{x\in\GG}e^{-\epsilon|x|}G(x)\Gt^*(x)
\,|\langle w,\pi(x)v\rangle|^2\\
=\frac1{(F_+,F_-)}\left(
\langle\mu_+(G)w,w\rangle\overline{\langle \mu_-(\Gt)v,v\rangle} +
\langle\mu_-(G)w,w\rangle\overline{\langle\mu_+(\Gt)v,v\rangle}\right)
\period
\end{multline}

\begin{remark}\label{drop-finite}  The limit on the left hand side
  of~\eqref{schurbdry} remains the same if we omit any finite number
  of values of~$x$. For instance, we can restrict the sum to
  $x\in\Gamma$ with $|x|\geq N$.
\end{remark}

\begin{lemma}\label{lemma-key8}
 Let $v\in\mcH^\infty$ and $w\in\mcH$. Then there exists a
 constant $C(v,w)$ depending only on $v$ and $w$, such that
\begin{equation}
\epsilon\sum_{x\in\GG}e^{-\epsilon|x|}\left|\langle
w,\pi(x)v\rangle\right|^2\leq
C(v,w)
\end{equation}
whenever $0<\epsilon\leq 1$.
\end{lemma}
\begin{proof}
By Corollary~\ref{gvb-4} there exists a constant $C$ depending only on
$v$ 
such that 
\begin{equation*}
\sum_{|x|=n}
\left|\langle
w,\pi(x)v\rangle\right|^2\leq C(v)\|w\|^2\period
\end{equation*}
Multiply the above inequality by $e^{-\epsilon n}$ and add up the
geometric
series to get the result.
\end{proof}
\begin{corollary}\label{cor-key8}
Let $v\in\mcH^\infty$ and $w\in\mcH$. Let $H$ be any function in
$C(\GG\sqcup\bGG)$
such that $\norm{H}_\infty\leq\delta$. Then there exists a
constant $C=C(v,w)$ such that
\begin{equation*}
\limsup_{\epsilon\to 0+}
\left|\epsilon\sum_{x\in\GG}e^{-\epsilon|x|}H(x)
\,|\langle w,\pi(x)v\rangle|^2\right|
\leq C\delta
\period
\end{equation*}
\end{corollary}

For $z\in\Gamma$ we defined $\lambda(z)$ as left-translation acting on
$C(\bGG)$. Here we will use the same notation for left-translation
acting on $C(\GG\sqcup\bGG)$.
\begin{lemma}\label{lemma-key-trasl}
Fix $G,\tilde G\in C(\GG\sqcup\bGG)$. Suppose that~\eqref{schurbdry}
holds for that~$G$ and~$\tilde G$ together with any $w\in\mcH$ and
$v\in\mcH^\infty$. Let~$z\in\GG$. Then \eqref{schurbdry} also holds if
we replace~$G$ by $\lambda(z)G$ or if we replace $\tilde G$ with
$\lambda(z)\tilde G$.
\end{lemma}
\begin{proof}
When we replace~$G$ with~$\lambda(z)G$ the left hand side
of~\eqref{schurbdry} becomes
\begin{align*}
\text{LHS}&=\lim_{\epsilon\to 0+}
\epsilon\sum_{x\in\GG}e^{-\epsilon|x|}G(z^{-1}x)\Gt^*(x)
\,|\langle w,\pi(x)v\rangle|^2\\
&=\lim_{\epsilon\to 0+}
\epsilon\sum_{x\in\GG}e^{-\epsilon|zx|}G(x)\Gt^*(zx)
\,|\langle w,\pi(zx)v\rangle|^2\\
&=\lim_{\epsilon\to 0+}
\epsilon\sum_{x\in\GG}e^{-\epsilon|zx|}G(x)\Gt^*(zx)
\,|\langle \pi(z^{-1})w,\pi(x)v\rangle|^2\\
&=\lim_{\epsilon\to 0+}
\epsilon\sum_{x\in\GG}e^{-\epsilon|x|}G(x)\Gt^*(x)
\,|\langle \pi(z^{-1})w,\pi(x)v\rangle|^2+\text{``vanishing term''}\\
&=\frac1{(F_+,F_-)}\biggl(
  \langle\mu_+(G)\pi(z^{-1})w,\pi(z^{-1})w \rangle
  \overline{\langle \mu_-(\Gt)v,v\rangle}\biggr. \\
&\qquad\qquad\qquad +\biggl.\langle\mu_-(G)\pi(z^{-1})w,\pi(z^{-1})w
    \rangle
    \overline{\langle\mu_+(\Gt)v,v\rangle}\biggr) \\
&=\frac1{(F_+,F_-)}\biggl(
  \langle\mu_+(\lambda(z)G)w,w \rangle
  \overline{\langle \mu_-(\Gt)v,v\rangle}\biggr. \\
    &\qquad\qquad\qquad +\biggl.\langle\mu_-(\lambda(z)G)w,w \rangle
    \overline{\langle\mu_+(\Gt)v,v\rangle}\biggr)
=\text{RHS.}
\end{align*}

It remains to show that the ``vanishing term'' vanishes. Write it as:
\begin{equation*}
\lim_{\epsilon\to 0+}
\epsilon\sum_{x\in\GG}e^{-\epsilon|x|}G(x)(e^{\epsilon(|x|-|zx|)}\Gt^*(zx)-\Gt^*(x))
\,|\langle \pi(z^{-1})w,\pi(x)v\rangle|^2 \period
\end{equation*}
Note that $e^{-\epsilon|z|}\leq e^{\epsilon(|x|-|zx|)}\leq e^{\epsilon
  |z|}$, so $|e^{\epsilon(|x|-|zx|)}-1|\leq e^{\epsilon |z|}-1$. Fix
any~$\delta>0$. Since $\Gt$~is continuous on $\GG\sqcup\bGG$ we can
choose $N>0$ so that
$|\Gt^*(zx)-\Gt^*(x)|=|\Gt(x^{-1}z^{-1})-\Gt(x^{-1})|\leq\delta$ when
$|x|\geq N$. According to Remark~\ref{drop-finite} we can omit all
terms where $|x|<N$. Then
\begin{equation*}
  |G(x)(e^{\epsilon(|x|-|zx|)}\Gt^*(zx)-\Gt^*(x))|
  \leq \norm{G}_\infty\bigl((e^{\epsilon|z|}-1)\norm{\Gt}_\infty+\delta\bigr) \period
\end{equation*}
Hence Corollary~\ref{cor-key8} says
\begin{multline*}
\limsup_{\epsilon\to 0+}
\left|\epsilon\sum_{x\in\GG}e^{-|x|}G(x)(e^{\epsilon(|x|-|zx|)}\Gt^*(zx)-\Gt^*(x))
\,|\langle \pi(z^{-1})w,\pi(x)v\rangle|^2\right| \\
\leq C(v,w)\,\norm{G}_\infty\,\delta \period
\end{multline*}
As $\delta>0$ is arbitrary the term does indeed vanish.

A strictly analogous calculation takes care of the case when we
replace $\Gt$ with~$\lambda(z)\Gt$. Note that in this second case one
uses Lemma~\ref{l:good-stable}, the $\Gamma$-invariance of~$\mcH^\infty$.
\end{proof}

\begin{definition} Recall that $\one_x$ was defined as
  $\one_{\bGG(x)}\in C(\bGG)$. Abusing notation, we will also let
  $\one_x=\one_{\GG(x)\sqcup\bGG(x)}\in C(\GG\sqcup\bGG)$.  Thus
  \begin{equation*}
    \one_x(\text{reduced word})=
    \begin{cases}
      1 & \text{if it starts with the reduced word
        for~$x$} \\
      0 & \text{otherwise}
    \end{cases}
  \end{equation*}
  for both finite and infinite reduced words.
\end{definition}

We proceed to prove \eqref{schurbdry} for $G=\one_a$ and
$\Gt=\one-\one_b$:

\begin{proposition}\label{limite-con-ab}
For $v\in\mcH^\infty$ and $w\in\mcH$ one has
\begin{multline}\label{somme-con-ab-bdry}
  \lim_{\epsilon\to 0+}\epsilon\sum_{x\in\GG}
  e^{-\epsilon|x|}\one_a(x)(\one-\one_b)(x^{-1})
  \,|\langle w,\pi(x)v\rangle|^2 \\
   =\frac{1}{(F_+,F_-)}
  \left(\langle\mu_+(\one_a)w,w\rangle
 \langle \mu_{-}(\one-\one_b)v,v\rangle  
\right. \\
  + \left.
  \langle\mu_-(\one_a)w,w\rangle\langle\mu_+(\one-\one_b)v,v\rangle
\right)\period
\end{multline}
\end{proposition}
\begin{proof}
Observe that the left hand side of~\eqref{somme-con-ab-bdry} is
equivalent to: 
\begin{equation}\label{somme-con-ab}
  \lim_{\epsilon\to 0+}\epsilon\sum_{x\in\GG(a),\; x\notin\tilde\GG(b^{-1})}
  e^{-\epsilon|x|} \,|\langle w,\pi(x)v\rangle|^2 \period
\end{equation}

As in the proof of Theorem~\ref{duplicity} define
\begin{align*}
    (F_0)_{ b}&=v\otimes\bar v & (F_0)_c&=0 \qquad\text{for $c\neq
    b$.}
\end{align*}
As per Proposition~\ref{DupLimP} one has
\begin{equation}\label{somme}
\wklim_{\epsilon\to 0+} \epsilon\sum_{n\geq 0} e^{-\epsilon n} \mcT^n
  F_0 = \frac{(F_0,F_-)F_++(F_0,F_+)F_-}{(F_+,F_-)}
  \period
\end{equation}
Dropping the $n=0$~term from the sum on the left has no effect on the
limit.  Consider the left-hand side of~\eqref{somme}. For~$n\geq 1$
Lemma~\ref{potenzaT} gives
\begin{align*}
  (\mcT^n F_0)_a=\sum_{c}(\mcT^n)_{a,c}(F_0)_c&=
    \sum_{\substack{x\in\Gamma\st |x|=n\\
        x\in\Gamma(a)\st x\notin\tilde\Gamma(b^{-1})}}
  P(x)(v\otimes \bar v)\\
&=  \sum_{\substack{x\in\Gamma\st |x|=n\\
        x\in\Gamma(a)\st x\notin\tilde\Gamma(b^{-1})}}
  \pi(x)v\otimes \overline{\pi(x)v}
  \period
\end{align*}
On applying the $a$-th component of the left-hand side
of~\eqref{somme} to the vector~$w$ and then calculating the inner
product with~$w$, one obtains
\begin{equation*}
\begin{aligned}
&\lim_{\epsilon\to 0+} \quad
\epsilon\;\langle\left(\sum_{n\geq1}  e^{-\epsilon n} \mcT^n F_0\right)_aw,w\rangle\\
=&\lim_{\epsilon\to 0+} \quad \epsilon\;
 \sum_{n\geq 1}e^{-\epsilon n}\left( \sum_{\substack{x\in\Gamma\st |x|=n\\
        x\in\Gamma(a)\st x\notin\tilde\Gamma(b^{-1})}}
  \langle w,\pi(x)v\rangle\overline{\langle w,\pi(x)v\rangle
  }\right)
\end{aligned}
\end{equation*}
which is equal to~\eqref{somme-con-ab}.

Now let us compute the right-hand side of~\eqref{somme}:
\begin{multline*}
  (F_0,F_\pm)=\sum_{c,d\st c\neq d}\TR((F_0)_c,(F_\pm)_d)=
  \sum_{d\st d\neq b}\TR(v\otimes\bar{v},(F_\pm)_d)\\
  =\sum_{d\st d\neq b}
  \langle\mu_\pm(\one_d)v,v\rangle=
  \langle\mu_\pm(\one-\one_b)v,v\rangle\;.
\end{multline*}  
Hence the $a$-th component of the right-hand side of~\eqref{somme} is
given by
\begin{equation}\label{sommedestre}
\frac{ \langle\mu_-(\one-\one_b)v,v\rangle \mu_+(\one_a)+
  \langle\mu_+(\one-\one_b)v,v\rangle \mu_-(\one_a)}{(F_+,F_-)}\;.
\end{equation}
On applying this operator to the vector~$w$ and then taking the inner
product with~$w$, one gets the right-hand side
of~\eqref{somme-con-ab-bdry}.
\end{proof}

\begin{corollary}\label{limite-con-zy0}
For every $v\in\mcH^\infty$, $w\in\mcH$ and
$y,z\in\GG$
one has
  \begin{multline}\label{schur}
    \lim_{\epsilon\to
      0+}\epsilon\;\sum_{x\in\Gamma}\one_z(x)\one_y^*(x) \,
    |\langle w,\pi(x)v\rangle|^2 \\ = \frac1{(F_+,F_-)} \left(
    \langle\mu_+(\one_z)w,w\rangle
    \overline{\langle\mu_-(\one_y)v,v\rangle}
    +\langle\mu_-(\one_z)w,w\rangle
    \overline{\langle\mu_+(\one_z)v,v\rangle}
    \right) \period
  \end{multline}
\end{corollary} 

\begin{proof}
Since $\lambda(b^{-1})(\one-\one_b)=\one_{b^{-1}}$ \eqref{schur}
follows from Proposition~\ref{limite-con-ab}
 and Lemma~\ref{lemma-key-trasl} for $|z|=|y|=1$. The general case
 follows again from Lemma~\ref{lemma-key-trasl}:
write $z=wc$
for the reduced word for $z$ and apply $\lambda(w)$ to $\one_c$,
analogously for $\one_y$.
\end{proof}

Now we finish the proof of~\eqref{schurbdry}, and so of~\eqref{schur2}
and of Theorem~\ref{Schur}.  The linear span of the functions
$\{\one_x\}_{x\in\GG}$ together with all finitely supported functions
on~$\GG$ is dense in~$C(\GG\sqcup\bGG)$.  From
Corollary~\ref{limite-con-zy0} one deduces that \eqref{schurbdry}
holds for~$G$ and $\Gt$~in this dense subset.  Finally,
Corollary~\ref{cor-key8} allows us to pass from the dense subset to
all of~$C(\GG\sqcup\bGG)$.

\subsection{Oddity} Here the reigning hypotheses are as follows:
\begin{itemize}
\item $\pi:\Gamma\to\mcU(\mcH)$ is a unitary representation of~$\Gamma$.
\item $\pi':\cp\to\mcB(\mcH')$ is an irreducible representation of~$\cp$.
\item $\iota:\mcH\to\mcH'$ is an imperfect boundary realization
  of~$\pi$.
\item As per Section~\ref{imfS}, $\mu$ and~$F$ are associated to~$\iota$.
\item The~\eqref{FTC} holds for~$F$.
\end{itemize}

Let $\mcH^\infty\subset\mcH$ be the dense subset of good vectors
relative to~$\iota$.
\begin{theorem}\label{th:schur} Let $G,\Gt\in C(\GG\sqcup\bGG)$;
  $v_1,v_2\in\mcH$; $v_3,v_4\in\mcH^\infty$. Then
\begin{multline}\label{schurbdrybis-odd}
\lim_{\epsilon\to 0+}
\epsilon\sum_{x\in\GG}e^{-\epsilon|x|}G(x)\Gt^*(x)
\langle v_1,\pi(x)v_3\rangle \overline{\langle v_2,\pi(x)v_4\rangle }=\\
\frac1{(F,F)}\left(
  \langle \pi'(G)\iota v_1,\iota v_2\rangle
  \overline{\langle\pi'(\Gt)\iota v_3,\iota v_4\rangle} 
\right)\period
\end{multline}
\end{theorem}
The proof is analogous to the proof in the previous subsection and we
omit it.

\bibliographystyle{amsalpha}
\bibliography{hyp0nw1}

\def\cprime{$'$}
\providecommand{\bysame}{\leavevmode\hbox to3em{\hrulefill}\thinspace}
\providecommand{\MR}{\relax\ifhmode\unskip\space\fi MR }
\providecommand{\MRhref}[2]{%
  \href{http://www.ams.org/mathscinet-getitem?mr=#1}{#2}
}
\providecommand{\href}[2]{#2}
\begin{thebibliography}{RSW89}

\bibitem[Ada94]{Adams}
S.~Adams, \emph{Boundary amenability for word hyperbolic groups and an
  application to smooth dynamics of simple groups}, Topology \textbf{33}
  (1994), no.~4, 765--783.

\bibitem[BG16]{BG}
A.~{Boyer} and {\L}.~{Garncarek}, \emph{{Asymptotic Schur orthogonality in
  hyperbolic groups with application to monotony}}, ArXiv e-prints (2016).

\bibitem[Dav96]{Davidson}
Kenneth~R. Davidson, \emph{{$C^*$}-algebras by example}, Fields Institute
  Monographs, vol.~6, American Mathematical Society, Providence, RI, 1996.
  \MR{1402012}

\bibitem[Dix64]{Dix}
J.~Dixmier, \emph{Les {$C^{\ast} $}-alg\`ebres et leurs repr\'esentations},
  Cahiers Scientifiques, Fasc. XXIX, Gauthier-Villars \& Cie,
  \'Editeur-Imprimeur, Paris, 1964.

\bibitem[Dix81]{Dix2}
Jacques Dixmier, \emph{von {N}eumann algebras}, North-Holland Mathematical
  Library, vol.~27, North-Holland Publishing Co., Amsterdam-New York, 1981,
  With a preface by E. C. Lance, Translated from the second French edition by
  F. Jellett. \MR{641217}

\bibitem[FTP82]{FT-P}
A.~Fig{\`a}-Talamanca and M.~A. Picardello, \emph{Spherical functions and
  harmonic analysis on free groups}, J. Funct. Anal. \textbf{47} (1982), no.~3,
  281--304.

\bibitem[FTS94]{FT-S}
A.~Fig{\`a}-Talamanca and T.~Steger, \emph{Harmonic analysis for anisotropic
  random walks on homogeneous trees}, Mem. Amer. Math. Soc. \textbf{110}
  (1994), no.~531, xii+68.

\bibitem[Gli61]{Glimm}
James Glimm, \emph{Type {I} {$C^{\ast} $}-algebras}, Ann. of Math. (2)
  \textbf{73} (1961), 572--612. \MR{0124756}

\bibitem[Hjo98]{Hjorth}
Greg Hjorth, \emph{When is an equivalence relation classifiable?}, Proceedings
  of the {I}nternational {C}ongress of {M}athematicians, {V}ol. {II} ({B}erlin,
  1998), no. Extra Vol. II, 1998, pp.~23--32. \MR{1648053}

\bibitem[KS96]{K-S1}
M.~G. Kuhn and T.~Steger, \emph{More irreducible boundary representations of
  free groups}, Duke Math. J. \textbf{82} (1996), no.~2, 381--436.

\bibitem[KS01]{K-S2}
\bysame, \emph{Monotony of certain free group representations}, J. Funct. Anal.
  \textbf{179} (2001), no.~1, 1--17.

\bibitem[KS04]{K-S3}
\bysame, \emph{Free group representations from vector-valued multiplicative
  functions. {I}}, Israel J. Math. \textbf{144} (2004), 317--341.

\bibitem[KSS16]{KSS}
M.~Gabriella Kuhn, Sandra Saliani, and Tim Steger, \emph{Free group
  representations from vector-valued multiplicative functions, {II}}, Math. Z.
  \textbf{284} (2016), no.~3-4, 1137--1162. \MR{3563271}

\bibitem[Kuh94]{Ku}
M.~G. Kuhn, \emph{Amenable actions and weak containment of certain
  representations of discrete groups}, Proc. Amer. Math. Soc. \textbf{122}
  (1994), no.~3, 751--757.

\bibitem[Pas01]{P}
W.~L. Paschke, \emph{Pure eigenstates for the sum of generators of the free
  group}, Pacific J. Math. \textbf{197} (2001), no.~1, 151--171.

\bibitem[PS86]{Py-Sw}
T.~Pytlik and R.~Szwarc, \emph{An analytic family of uniformly bounded
  representations of free groups}, Acta Math. \textbf{157} (1986), no.~3-4,
  287--309.

\bibitem[PS96]{Pe-S}
C.~Pensavalle and T.~Steger, \emph{Tensor products with anisotropic principal
  series representations of free groups}, Pacific J. Math. \textbf{173} (1996),
  no.~1, 181--202.

\bibitem[RSW89]{RSW}
Iain Raeburn, Allan~M. Sinclair, and Dana~P. Williams, \emph{Equivariant
  completely bounded operators}, Pacific J. Math. \textbf{139} (1989), no.~1,
  155--194. \MR{1010790}

\bibitem[Rud73]{Rudin}
Walter Rudin, \emph{Functional analysis}, McGraw-Hill Book Co., New
  York-D\"usseldorf-Johannesburg, 1973, McGraw-Hill Series in Higher
  Mathematics. \MR{0365062}

\bibitem[Sti55]{Sti}
W.~Forrest Stinespring, \emph{Positive functions on {$C^*$}-algebras}, Proc.
  Amer. Math. Soc. \textbf{6} (1955), 211--216. \MR{0069403 (16,1033b)}

\bibitem[Tak03]{T}
M.~Takesaki, \emph{Theory of operator algebras. {II}}, Encyclopaedia of
  Mathematical Sciences, vol. 125, Springer-Verlag, Berlin, 2003, Operator
  Algebras and Non-commutative Geometry, 6.

\bibitem[Yos51]{Yos}
Hisaaki Yoshizawa, \emph{Some remarks on unitary representations of the free
  group}, Osaka Math. J. \textbf{3} (1951), 55--63. \MR{0041854}

\end{thebibliography}

\end{document}